\documentclass[10.5pt,reqno]{amsart}
\usepackage{longtable}
\usepackage{hyperref}
\usepackage[T1]{fontenc}
\usepackage[utf8]{inputenc}
\usepackage[english]{babel}
\usepackage{textcomp}
\usepackage{dsfont}
\usepackage{latexsym}
\usepackage{amssymb}
\usepackage{amsthm}
\usepackage{amsmath}
\DeclareMathAlphabet{\mathpzc}{OT1}{pzc}{m}{en}
\usepackage{yfonts}
\usepackage{xfrac}
\usepackage{newlfont}
\usepackage{graphicx}
\usepackage{mathtools}
\usepackage{comment}
\usepackage{indentfirst}
\usepackage{braket}
\usepackage{mathrsfs}
\usepackage{fixmath}

\usepackage{scalerel}[2014/03/10]
\usepackage[usestackEOL]{stackengine}
\newcommand{\dashint}{\,\ThisStyle{\ensurestackMath{%
			\stackinset{c}{.2\LMpt}{c}{.5\LMpt}{\SavedStyle-}{\SavedStyle\phantom{\int}}}%
		\setbox0=\hbox{$\SavedStyle\int\,$}\kern-\wd0}\int}
	
\usepackage{xcolor}

\DeclareMathOperator{\card}{Card}

\DeclareMathOperator{\pr}{pr}

\DeclareMathOperator{\tr}{Tr}

\DeclareMathOperator{\Aut}{Aut}

\newcommand{\Aff}{\mathrm{Aff}}
\newcommand{\Hol}{\mathrm{Hol}}

\newcommand{\ee}{\mathrm{e}}

\newcommand{\vect}[1]{\mathbf{{#1}}}
\newcommand{\dd}{\mathrm{d}}

\DeclarePairedDelimiter{\abs}{\lvert}{\rvert}

\DeclarePairedDelimiter{\norm}{\lVert}{\rVert}

\let\originalleft\left
\let\originalright\right
\renewcommand{\left}{\mathopen{}\mathclose\bgroup\originalleft}
\renewcommand{\right}{\aftergroup\egroup\originalright}

\newcommand{\Ms}{\mathscr{M}}

\newcommand{\N}{\mathds{N}}
\newcommand{\Z}{\mathds{Z}}

\newcommand{\C}{\mathds{C}}
\newcommand{\Hd}{\mathds{H}}
\newcommand{\Od}{\mathds{O}}
\newcommand{\R}{\mathds{R}}
\newcommand{\F}{\mathds{F}}

\newcommand{\T}{\mathds{T}}
\newcommand{\epsb}{\mathbold{\eps}}

\newcommand{\gf}{\mathfrak{g}}

\newcommand{\Hs}{\mathscr{H}}
\newcommand{\Us}{\mathscr{U}}

\newcommand{\Ac}{\mathcal{A}}
\newcommand{\Bc}{\mathcal{B}}
\newcommand{\Cc}{\mathcal{C}}
\newcommand{\Dc}{\mathcal{D}}

\newcommand{\Fc}{\mathcal{F}}
\newcommand{\Gc}{\mathcal{G}}
\newcommand{\Hc}{\mathcal{H}}
\newcommand{\Ic}{\mathcal{I}}

\newcommand{\Kc}{\mathcal{K}}
\newcommand{\Lc}{\mathcal{L}}
\newcommand{\cM}{\mathcal{M}}
\newcommand{\Nc}{\mathcal{N}}

\newcommand{\Pc}{\mathcal{P}}
\newcommand{\Qc}{\mathcal{Q}}
\newcommand{\Rc}{\mathcal{R}}

\newcommand{\Uc}{\mathcal{U}}
\newcommand{\Vc}{\mathcal{V}}
\newcommand{\Wc}{\mathcal{W}}

\renewcommand{\Im}{\mathrm{Im}\,}
\renewcommand{\Re}{\mathrm{Re}\,}

\newcommand{\meg}{\leqslant}
\newcommand{\Meg}{\geqslant}
\newcommand{\eps}{\varepsilon}
\renewcommand{\phi}{\varphi}
\newcommand{\mi}{\mu}

\newcommand{\Lin}{\mathscr{L}}

\title[Invariant Spaces of Holomorphic Functions]{Invariant Spaces of Holomorphic Functions on Symmetric Siegel domains}

\date{}

\begin{document}

\theoremstyle{definition}
\newtheorem{deff}{Definition}[section]

\newtheorem{oss}[deff]{Remark}

\newtheorem{ex}[deff]{Example}

\newtheorem{nott}[deff]{Notation}

\theoremstyle{plain}
\newtheorem{teo}[deff]{Theorem}

\newtheorem{lem}[deff]{Lemma}

\newtheorem{prop}[deff]{Proposition}

\newtheorem{cor}[deff]{Corollary}

\author[M. Calzi]{Mattia Calzi}

\address{Dipartimento di Matematica, Universit\`a degli Studi di
	Milano, Via C. Saldini 50, 20133 Milano, Italy}
\email{{\tt mattia.calzi@unimi.it}}

\thanks{{\em Math Subject Classification 2020} 46E15, 47B33, 32M15 }
\thanks{The author is a member of the 	Gruppo Nazionale per l'Analisi Matematica, la Probabilit\`a e le	loro Applicazioni (GNAMPA) of the Istituto Nazionale di Alta	Matematica (INdAM).   The author was partially funded by the INdAM-GNAMPA Project CUP\_E55F22000270001. 
} 

\begin{abstract}
	In this paper we consider a symmetric Siegel domain $D$ and some natural representations of the M\"obius group $G$ of its biholomorphisms and of the group $\Aff$ of its affine biholomorphisms. We provide a  classification of  the affinely-invariant semi-Hilbert spaces (satisfying some natural additional assumptions) on tube domains,  and improve the classification of M\"obius-invariant  semi-Hilbert spaces on general domains.
\end{abstract}
\maketitle

\section{Introduction}

In~\cite{ArazyFisher}, Arazy and Fisher showed that the classical Dirichlet space on the unit disc $D$ in $\C$, namely
\[
\Dc=\Set{f\in \Hol(D)\colon \int_{D} \abs{f'(z)}^2\,\dd z<\infty},
\]
where $\Hol(D)$ denotes the space of holomorphic functions in $D$,
is the unique M\"obius-invariant semi-Hilbert space of holomorphic functions on $D$ which embeds continuously into the Bloch space, namely
\[
\Bc\coloneqq \Set{f\in \Hol(D)\colon \sup_{z\in D}(1-\abs{z}^2) \abs{f'(z)}<\infty },
\]
whose seminorm vanishes on constant functions, and for which the action of the M\"obius group (by composition) is continuous and bounded. 
This result was partially motivated by an earlier result by Rubel and Timoney~\cite{RubelTimoney}, which characterized the Bloch space  $\Bc$ as the largest `decent' M\"obius-invariant space of holomorphic functions on $D$. Here, we say that a semi-Banach  space $X$ of holomorphic functions on $D$ is decent if there is a  continuous linear functional $L$ on $\Hol(D)$ which induces a non-zero continuous linear functional on $X$. More precisely,  if $X$ is a decent space of holomorphic functions in which  composition with the elements of the (M\"obius) group of biholomorphisms of $D$, namely
\[
G=\Set{ z\mapsto \alpha\frac{z-b}{1-\overline b z}\colon \alpha\in \T, \abs{b}<1 },
\]
induce a bounded representation of $G$, then $X\subseteq \Bc$ continuously. 

The characterization of the Dirchlet space by M\"obius invariance was later extended to the Dirichlet space on the unit ball $D$ in $\C^n$, for isometric invariance, by Peetre in an unpublished note~\cite{Peetre}, and then Zhu in~\cite{Zhu}. See also~\cite{Peloso,Arazy3} for other descriptions of this space, and~\cite[Theorem 5]{ArazyFisher2} for the proof of uniqueness under the assumption of `bounded' invariance (that is, under the assumption that the group of biholomorphisms of $D$ acts boundedly by composition). 

This kind of results have also been considered in more general contexts, such as that of (irreducible) bounded symmetric domains. We recall that  a bounded connected open subset $D$ of $\C^n$ is said to be a symmetric domain if for every $z\in D$ there is a holomorphic involution of $D$ having $z$ as its unique (or, equivalently, as an isolated) fixed point. The domain $D$ is then homogeneous. Namely, the `M\"obius' group, that is, the group of its biholomorphisms, acts transitively on $D$. 
The domain $D$ is said to be irreducible if it is not biholomorphic to a product of two non-trivial symmetric domains. 

To begin with, the maximality property of the Bloch space was  extended to general bounded symmetric domains  in~\cite{Timoney2}, using Timoney's generalization of the Bloch space, cf.~\cite{Timoney1}. Unfortunately, the main results of~\cite{Timoney2} are  incorrect (cf., also,~\cite{Agranovski,Survey}), since they imply (cf.~\cite[Corollary 0.2]{Timoney2}) that the only closed subspaces of $\Hol(D)$ which are invariant under composition with the biholomorphisms of $D$, where $D$ is an irreducible bounded symmetric domain, are $\Set{0}$, $\C \chi_D$, and $\Hol(D)$. As~\cite[Proposition 4.12 and the following remarks]{Arazy} show, this is not always the case. In fact, there are (irreducible bounded symmetric) domains on which Timoney's Bloch space embeds continuously in a strictly larger `decent' semi-Banach space (cf., e.g.,~\cite[Theorem 1.3]{Garrigos}).

Returning to the hilbertian setting, also more general M\"obius-invariant spaces on an irreducible symmetric domain $D$ were investigated. Let $\widetilde G$ be the universal covering of the component of the identity $G_0$ of the group $G$ of  biholomorphisms of $D$, and consider the representation $\widetilde U_\lambda$ of $\widetilde G$ in $\Hol(D)$ defined, for every $\lambda\in \R$, by
\[
\widetilde U_\lambda(\phi) f \coloneqq (f\circ \phi^{-1}) (J\phi^{-1})^{\lambda/p},
\]
for every $\phi \in \widetilde G$ and for every $f\in \Hol(D)$, where $ \widetilde G$ acts on $D$ through the canonical projection $\widetilde G\to G_0$, $p$ is the genus of $D$, $J\phi=\det_\C\phi'$ is the (complex) Jacobian of $\phi$ (considered as a biholomorphism of $D$), and $(J\phi)^{-\lambda/p}=\ee^{-(\lambda/p) \log J( \phi,\,\cdot\,)}$, where $\log J$ is the unique continuous function on $\widetilde G\times D $ satisfying $\log J(e, 0)=0$ and $\ee^{\log J(\phi,z)}=(J \phi)(z)$.\footnote{Here, we assume that $D$ is in its cirular convex realization, so that $0\in D$. Observe that $\log J$ is well defined since $\widetilde G\times D$ is simply connected.}
Then, it is clear that the unweighted Bergman space 
\[
A^2(D)\coloneqq\Hol(D)\cap L^2(D)
\]
is $\widetilde U_p$-invariant with its norm. Since it embeds continuously into $\Hol(D)$, it is a reproducing kernel Hilbert space. Denote by $\Kc$ its reproducing kernel, so that $\Kc(\,\cdot\,,z)\in A^2(D)$ and
\[
f(z)=\langle f\vert \Kc(\,\cdot\,,z)\rangle_{A^2(D)}
\]
for every $f\in A^2(D)$ and for every $z\in D$. As~\cite{VergneRossi} shows, $\Kc^{\lambda/p}$ is the reproducing kernel of  a  (necessarily $\widetilde U_\lambda$-invariant with its norm) reproducing kernel Hilbert space  if and only if $\lambda$ belongs to the so-called Wallach set, which is $\Set{j a/2\colon j=0,\dots, r-1}\cup (a(r-1)/2, +\infty)$ for  suitable   $a,r\in\N$ (cf.~Definition~\ref{def:4}). In particular, $r$ denotes the rank of $D$. In the same paper, a description of the aforementioned spaces was provided on the (unbounded) realization of $D$ as a Siegel domain.
The preceding spaces were proved to be the unique reproducing kernel Hilbert spaces of holomorphic functions on $D$ on which $\widetilde U_\lambda$ induces a bounded representation (satisfying some  continuity assumptions) in~\cite{ArazyFisher3} when $D$ is the unit disc in $\C$ and the action is isometric, and in~\cite[Theorem 3]{ArazyFisher2} in the general case. This kind of analysis was later developed also on bounded homogeneous domains (cf.~\cite{Ishi6}) and in homogeneous Siegel domains (cf.~\cite{Ishi3,Ishi4,Ishi5}).

In addition, also Dirichlet-type $\widetilde U_\lambda$-invariant spaces were considered. It was proved that, when $D$ is the unit ball in $\C^n$ (that is, when the rank $r$ of $D$ is $1$), then there are non-trivial non-Hausdorff semi-Hilbert subspaces $H$ of $\Hol(D)$ in which $\widetilde U_\lambda$ induces a bounded representation satisfying some form of continuity, if and only if $\lambda\in - \N$, and that there is only one such space, up to isomorphisms: see~\cite{Peetre2} for the unit disc in $\C$; see~\cite{Peetre} and~\cite{Zhu}  for the case $\lambda=0$, as mentioned earlier, and for isometric invariance; see~\cite[Theorems 2 and 5]{ArazyFisher2} and also~\cite[Theorem 5.2]{Arazy} for the case of isometric invariance and for the general case when $\lambda=0$, and~\cite[Theorem 5.3]{Rango1} for the general case.

For domains $D$ of higher rank, the situation is  more complicated, and the study of this problem is largely based on the  decomposition of the space of polynomials on $D$ into mutually inequivalent irreducible subspaces under the action of the group of linear automorphisms of $D$ (which is a maximal compact subgroup of the group  $G$ of biholomorphisms of $D$ when $D$ is in its circular convex relatization), cf.~\cite{FarautKoranyi2}. The existence and uniqueness problem has been  completely solved, even though the resulting spaces do  not always have a clear description, especially on Siegel domains which are not of tube type: cf.~\cite{ArazyFisher3} and~\cite[Theorem 5.2]{Arazy} for isometric invariance, and Theorems~\ref{teo:4} and~\ref{teo:5} below for the general case. 

Let us also mention that there is a number of papers where  the (scalar products of the) preceding spaces are described in terms of integral formulas involving suitably\footnote{In fact, invariance is only required under the action of a suitable subgroup of $G_0$, which is not always the same.} invariant differential operators. See~\cite{Arazy2,ArazyUpmeier,ArazyUpmeier2} for irreducible bounded symmetric domains of tube type,~\cite{Peloso,Arazy3} for the case of the unit ball in $\C^n$, and~\cite{Yan} for general irreducible bounded symmetric domains. See also~\cite{Garrigos} for irreducible symmetric tube domains (that is, tube type domains in their unbounded  realization as Siegel domains) and~\cite{Arcozzietal} for the Siegel upper half-space, that is, the Siegel domain corresponding the unit ball in $\C^n$. 

Finally, we also mention that other classes of invariant spaces have been investigated, satisfying suitable minimality or maximality properties. See~\cite{ArazyFisher4,Peloso,ArazyFisherPeetre,Zhu,ArazyUpmeier3,AlemanMas,Survey} to name but a few.

\medskip

In this paper we consider the above and some related problems. We shall deal with the realization of $D$ as a Siegel domain of type II, so that
\[
D=\Set{(\zeta,z)\in E\times F_\C\colon \Im z-\Phi(\zeta)\in \Omega},
\]
where $E$ is a complex Hilbert space of dimension $n$, $F$ is a real Hilbert space of dimension $m$, $F_\C$ is its complexification, $\Omega$ is an open convex cone not containing affine lines in $F$, $\Phi\colon E\times E\to F_\C$ is a non-degenerate $\overline{\Omega}$-positive hermitian map, and $\Phi(\zeta)=\Phi(\zeta,\zeta)$ for every $\zeta\in E$. 

After recalling some basic facts and notation, we shall consider the problem of classifying all $\Aff$-$\Uc_{\lambda}$-invariant semi-Hilbert spaces of holomorphic functions on $D$, where $\Uc_\lambda$ is defined by
\[
\Uc_\lambda(\phi) f =(f\circ \phi^{-1}) \abs{J\phi^{-1}(0,0) }^{\lambda/p}
\]
for every $\phi\in \Aff$ and for every $f\in \Hol(D)$.
We shall assume that $H$ satisfies a suitable strenghtening of the decency hypotheses considered by Rubel and Timoney~\cite{RubelTimoney}, which we shall call `strong decency'. Namely, we say that $H$ is strongly decent if the space of continuous linear functionals on $H$ which extend to continuous linear functionals on $\Hol(D)$ is dense in $H'$ (in the weak dual topology, or, equivalently, in the strong dual topology). 
This is equivalent to saying that there is a closed subspace $V$ of $\Hol(D)$ such that $H\cap V$ is the closure of $\Set{0}$ in $H$ and the canonical mapping $H\to \Hol(D)/V$ is continuous (cf.~Proposition~\ref{prop:6}). On the one hand, this requirement is analogous to the   assumptions considered in~\cite{ArazyFisher2,Arazy} to deal with the bounded case (and M\"obius invariance), as we shall see in Remark~\ref{oss:1}. On the other hand, even in the $1$-dimensional case it is not clear to us whether the simple decency assumption is sufficient to prevent some algebraic issues that may occur when classifying $\Aff$-$\Uc_\lambda$-invariant spaces (and even $\widetilde G $-$\widetilde U_\lambda$-invariant spaces, in some cases). See~\cite[Section 4]{Rango1} for a lenghtier discussion of these issues.

When $D$ is a tube domain, we are then able to provide a complete classification of the above mentioned spaces using the description of $G(\Omega)$-invariant irreducible subspaces of the space of polynomials on $F$ provided in~\cite[Theorem XI.2.4]{FarautKoranyi}, where $G(\Omega)$ denotes the group of linear automorphisms of $\Omega$, combined with a description of a related class of mean-periodic functions provided in~\cite[Proposition 7.1]{Rango1}.  
For the case of Siegel domains of rank $1$, that is, those corresponding to the unit ball in $\C^{n+1}$, see~\cite{Rango1}.

We then pass to M\"obius-invariant spaces and describe, when $D$ is a tube domain, which of the preceding $\Aff$-$\Uc_\lambda$-invariant spaces are actually $\widetilde G$-$\widetilde U_\lambda$-invariant (cf.~Theorems~\ref{teo:2} and~\ref{teo:4}), thus extending~\cite{Garrigos} in the setting of Siegel domains. For what concerns more general Siegel domains, we are only able to obtain partial results, even though we are able to strengthen  the known uniqueness results (cf.~Theorem~\ref{teo:5}).

Concerning our methods, the techniques applied to deal with affinely-invariant spaces on tube domains  seem to be new, up to some extent, and are essentially based on the study of the zero locus of the seminorm. The study of M\"obius-invariant spaces is largely based on the previous works on the subject (cf., e.g.,~\cite{Garrigos} for tube domains and~\cite{ArazyFisher2,Arazy} for general domains), combined with our results on tube domains.

\medskip

Here is a plan of the paper. In Section~\ref{sec:1}, we shall collect several basic definitions and facts concerning homogeneous Siegel domains of type II and their groups of automorphisms, as well as establish our notation. Among the various algebraic descriptions of symmetric  cones, we shall generally stick to that of Jordan algebras (cf.~\cite{FarautKoranyi}); for simplicity, we shall avoid the formalism of Jordan triple systems (cf., e.g.,~\cite{Loos}) and refer to specific results when we need more information on symmetric domains which are not of tube type.  We also collect some remarks on reproducing kernel Hilbert spaces and recall the definition and some basic properties of strongly decent and saturated spaces.

In Section~\ref{sec:2}, we shall describe some known results on the $G_T$-$\Uc_\lambda$-invariant reproducing kernel Hilbert spaces of holomorphic functions on $D$, where $G_T$ is a simply transitive triangular group of affine biholomorphisms of $D$. We shall then apply these results in order to deal with $\Aff$-$\Uc_\lambda$-invariant semi-Hilbert spaces on (irreducible symmetric) tube domains. 
In Section~\ref{sec:4}, we shall deal with M\"obius-invariant spaces on general (irreducible symmetric) Siegel domains.

\section{Preliminaries}\label{sec:1}

\subsection{General Notation}\label{sec:1:1}

Throughout the paper, $E$ will denote a complex Hilbert space of dimension $n\Meg 0$, $F$ a real Hilbert space of dimension $m>0$, and $F_\C$ its complexification. We shall denote by $\Omega$ a  symmetric cone in $F$, that is, an open convex cone which does not contain affine lines, has a transitive group of linear automorphisms, and is self-dual with respect to the scalar product of $F$, that is, 
\[
\Omega=\Set{x\in F\colon \forall y\in \overline{\Omega}\setminus \Set{0}\:\: \langle x,y\rangle>0}.
\] 
We shall also assume that $\Omega$ is irreducible, that is, that $\Omega$ cannot be decomposed as the product of two (non-trivial) symmetric cones. Finally, we shall denote with $\Phi\colon E\times E\to F_\C$ a non-degenerate $\overline{\Omega}$-positive  hermitian mapping such that the Siegel domain
\[
D=\Set{(\zeta,z)\in E\times F_\C\colon \Im z-\Phi(\zeta)\in \Omega},
\]
where $\Phi(\zeta)=\Phi(\zeta,\zeta)$ for simplicity, is symmetric. In other words, for every $(\zeta,z)\in D$ there is an involutive biholomorphism $\iota$ of $D$ such that $(\zeta,z)$ is an isolated fixed point of $D$. Notice that $D$ is then homogeneous (cf.~\cite[No.\ 17]{Cartan}), that is, has a transitive group of biholomorphisms. In addition, since $\Omega$ is assumed to be irreducible, also $D$ is irreducible (cf.~\cite[Corollary 4.8]{Nakajima}), that is, $D$ is not biholomorphic to a product of two (non-trivial) symmetric Siegel domains. We shall denote by $e_\Omega$ a fixed point of $\Omega$.

It is then known that the group $\Aff$ of affine automorphisms of $D$ acts transitively on $D$ (cf.~\cite[Theorem 7.3]{Murakami}). In addition, $\Nc=E\times F$, endowed with the product defined by
\[
(\zeta,x)(\zeta',x')=(\zeta+\zeta', x+x'+2\Im \Phi(\zeta,\zeta')),
\]
becomes a $2$-step nilpotent Lie group with centre $F$, and acts freely and faithfully on $E\times F_\C$ and $D$ by affine transformations. Namely,
\[
(\zeta,x)\cdot (\zeta',z')=(\zeta+\zeta', z'+x+i\Phi(\zeta)+2 i \Phi(\zeta',\zeta))
\]
for every $(\zeta,x)\in \Nc$ and for every $(\zeta',z')\in E\times F_\C$.
Identifying $\Nc$ with a subgroup of $\Aff$, it then follows that $\Nc$ is a closed normal subgroup of $\Aff$ and that $\Aff$ is the semi-direct product of $\Nc$ and the group $GL(D)$ of linear automorphisms of $D$. Notice that
\[
GL(D)=\Set{A\times B_\C\colon A\in GL(E),  B\in G(\Omega), B_\C \Phi=\Phi(A\times A)},
\]
where $G(\Omega)$ denotes the group of linear automorphisms of $\Omega$ and $B_\C=B\otimes_\R \C$ (cf.~\cite[Propositions 2.1 and 2.2]{Murakami}).

\subsection{Symmetric Cones}\label{sec:Sym}

In this subsection, we recall some basic aspects of  the theory of (irreducible) symmetric cones, and describe some examples.

\begin{deff}
	A (real or complex) Jordan algebra is a commutative, not necessarily associative (real or complex) algebra $A$ such that $x^2(xy)=x(x^2 y)$ for every $x,y\in A$.
	A real Jordan algebra $A$ is said to be Euclidean if it is endowed with a scalar product such that $\langle x y\vert z\rangle=\langle y\vert x z\rangle$ for every $x,y,z\in A$.
\end{deff}

See~\cite{FarautKoranyi} for a more detailed study of (Euclidean) Jordan algebras and a proof of the following result (Theorems III.2.1 and III.3.1 of the cited reference).

\begin{prop}
	If $A$ is a finite-dimensional real Euclidean Jordan algebra with identity $e$, then the interior $S(A)$ of $\Set{x^2\colon x\in A} $ is a symmetric cone in $A$. 
	
	Conversely, if $C$ is a symmetric cone in $F$ and $e\in C$, then there is a Euclidean Jordan algebra structure on $F$, with identity $e$ and the same scalar product, such that $C=S(F)$.
\end{prop}

Therefore, $F$ may be endowed with the structure of a Euclidean Jordan algebra with the same scalar product and identity $e_\Omega$, in sch a way that $S(F)=\Omega$. 
We shall then endow $F_\C$ with the complexification of the Jordan algebra structure of $F$, so that $F_\C$ is a (complex) Jordan algebra with identity $e_\Omega$.

Since $\Omega$ is assumed to be irreducible, $F$ is then a simple Jordan algebra, that is, $F$ does not contain  non-trivial ideals (cf.~\cite[Propositions III.4.4 and III.4.5]{FarautKoranyi}).
Finite-dimensional simple unital Euclidean real Jordan algebras may be classified, up to isomorphism (cf.~\cite[Corollary IV.1.5 and Theorem V.3.7]{FarautKoranyi}). We shall describe in   Examples~\ref{ex:1} and~\ref{ex:2} a class of representatives of all finite-dimensional simple unital Euclidean real Jordan algebra. Notice that this description is somewhat redundant.

\begin{ex}\label{ex:1}
	Take an integer $r\Meg 1$ and let $\F$ be either $\R$, $\C$, or the division ring of Hamilton quaternions $\Hd$. Then, the space $A$ of hermitian $r\times r$ matrices over $\F$, endowed with the symmetrized product $x\circ y= (x y+y x)/2$ and the scalar product $(x,y)\mapsto \Re \tr(xy)=\tr (x\circ y)$,  is a real Euclidean  Jordan algebra with identity $I_r=(\delta_{j,k})_{j,k=1,\dots, r}$. 
	The symmetric cone $S(A)$ is then the cone of non-degenerate positive hermitian $r\times r$ matrices over $\F$.
	The same holds if $r\meg 3$ and $\F$ is the division algebra of Cayley octionions $\Od$, even though this latter fact is more difficult to prove (cf.~\cite[Corollary V.2.6]{FarautKoranyi}).
	
	In particular, if $r=1$, then $A=\R$ with the usual structure, and $S(A)=(0,\infty)$.
\end{ex}

\begin{ex}\label{ex:2}
	Take an integer $k \Meg 1$, and let $A$ be the algebra of $2\times 2$ formal symmetric matrices of the form $\left(\begin{smallmatrix} a & b\\ b & c \end{smallmatrix}\right)$, with $a,c\in \R$ and $b\in \R^{k}$, endowed with the symmetrized product 
	\[
	\left(\begin{matrix} a & b\\ b & c \end{matrix}\right)\circ \left(\begin{matrix} a' & b'\\ b' & c' \end{matrix}\right)=\left(\begin{matrix} a a'+ \langle b,b'\rangle & (a b'+a'b+cb'+c' b)/2\\ (a b'+a'b+cb'+c' b)/2 & cc'+\langle b,b'\rangle \end{matrix}\right)
	\]
	and the scalar product $(x,y)\mapsto \tr(x\circ y) $. In other words, $\langle \left(\begin{smallmatrix} a & b\\ b & c \end{smallmatrix}\right), \left(\begin{smallmatrix} a & b\\ b & c \end{smallmatrix}\right)\rangle= a a'+2 \langle b,b'\rangle+c c'$. Then, $S(A)$ is the set of formally positive non-degenerate symmetric matrices on $A$, that is, the set of $\left(\begin{smallmatrix} a & b\\ b & c \end{smallmatrix}\right)$ with $a>0$ and $\abs{b}^2< ac$. 
	
	Notice that, when $k=1,2,4,8$, we identify $\R^k$ with $\R$, $\C$, $\Hd$, $\Od$, respectively, in such a way that $\langle b, b'\rangle=\Re (b\overline{b'})$, and we replace $  \left(\begin{smallmatrix} a & b\\ b & c \end{smallmatrix}\right)$ with $  \left(\begin{smallmatrix} a & b\\ \overline b & c \end{smallmatrix}\right)$, then we obtain the examples considered in Example~\ref{ex:1} for $r=2$.
\end{ex}

\begin{deff}
	Let $A$ be a (finite-dimensional) Jordan algebra over $\F=\R$ or $\C$ with identity $e$. An element $x$ of $ A$ is said to be invertible in $A$ if $x$ has a (necessarily unique) inverse in the associative subalgebra $\F[x]$ of $A$ generated by $x$ and $e$. We then define $x^{-1}$ as the inverse of $x$ in $\F[x]$.
	
	In addition, we define $\det_A(x)$ as the determinant of the mapping $\F[x]\ni y\mapsto x y\in \F[x]$. We call $\det_A$ the determinant polynomial of $A$.
\end{deff}

Notice that $\det_A(x)\neq 0$ if and only if $x$ is invertible in $A$, and that $\det_A(x)$ is the \emph{norm} of $x$ relative to the associative algebra $\F[x]$.

\begin{ex}\label{ex:3}
	If $A$ is as in Example~\ref{ex:1} (and $\F\neq \Od$), then $x\in A$ is invertible in $A$ if and only if it is invertible as a matrix, in which case the inverse of $x$ in $A$ is the inverse of $x$ as a matrix. This happens because the algebras generated by $x$ and $e$ in $A$ and in the algebra of $r \times r$ matrices coincide, and have the same product. In addition, $\det_A$ is the real determinant when $\F=\R$, and the complex determinant when $\F=\C$; when $\F=\Hd$ and $A$ is identified with a suitable algebra of skew-symmetric $(2r)\times (2 r)$ complex matrices (cf.~\cite[p.~88]{FarautKoranyi}), then $\det_A$ becomes the Pfaffian, possibly up to a unimodular constant which depends on the chosen identification (we provide no interpretations of $x^{-1}$ and $\det_A(x)$ when $r= 3$ and $\F=\Od$).
	
	If $A$ is as in Example~\ref{ex:2}, then $\left(\begin{smallmatrix} a & b\\ b & c \end{smallmatrix}\right)$ is invertible in $A$ if and only if $\det_A \left(\begin{smallmatrix} a & b\\ b & c \end{smallmatrix}\right)= ac- \abs{b}^2$ is non-zero, in which case
	\[
	\left(\begin{matrix} a & b\\ b & c \end{matrix}\right)^{-1}=\frac{1}{ac-\abs{b}^2} \left(\begin{matrix} c & -b\\ -b & a \end{matrix}\right).
	\]
\end{ex}

\begin{deff}
	Let $A$ be a (finite-dimensional) Jordan algebra with identity $e$. A Jordan frame in $A$ is a  family $(e_j)$ of non-zero idempotents of $A$ such that $e_j e_{j'}=0$ for every $j,j'$, $j\neq j'$, such that $\sum_j e_j=e$, and such that no $e_j$ can be written as a sum of two non-zero idempotents. The rank of $A$ is the common length of   its Jordan frames (cf.~\cite[Theorems III.1.1 and III.1.2]{FarautKoranyi}).
\end{deff}

\begin{deff}
	Let $(e_j)$ be a Jordan frame of a unital  Euclidean  real Jordan algebra $A$. Then, $A_j\coloneqq \Set{x\in A\colon (e_1+\cdots+e_j)x=x}$ is a Jordan subalgebra of $A$ with identity $e_{1}+\cdots+e_j$. Denote by $\pr_j\colon A\to A_j$ the orthogonal projector. We may then define the generalized power functions
	\[
	\Delta_{(e_1,\dots,e_r)}^{\vect s}\colon S(A)\ni x \mapsto ({\det}_{A_1} \pr_1(x))^{s_1-s_{2}}\cdots({\det}_{A_{r-1}} \pr_{r-1}(x))^{s_{r-1}-s_r}  ({\det}_{A_r} \pr_r(x))^{s_r}\in \C
	\]
	for every $\vect s\in \C^r$.
\end{deff}

\begin{ex}\label{ex:4}
	If $A$ is as in Example~\ref{ex:1}, then the idempotents $e_j\coloneqq (\delta_{p,j}\delta_{q,j} )_{p,q=1,\dots,r}$, $j=1,\dots,r$ form a Jordan frame of $A$. In particular, $A$ has rank $r$. When $\F=\R$ or $\C$, the corresponding functions $\det_{A_j}$ are  the  minors over $\F$ corresponding to the first $j$ rows and columns, thanks to Example~\ref{ex:3}.
	
	If $A$ is as in Example~\ref{ex:2}, then the idempotents $e_1=\left(\begin{smallmatrix} 1 & 0\\ 0 & 0 \end{smallmatrix}\right)$ and $e_2=\left(\begin{smallmatrix} 0 & 0\\ 0 & 1 \end{smallmatrix}\right)$ form a Jordan frame of $A$. In particular, $A$ has rank $2$. The corresponding function $\det_{A_1}$ is then simply the projection $\left(\begin{smallmatrix} a & b\\ b & c \end{smallmatrix}\right)\mapsto a$.
\end{ex}

\subsection{Riesz Distributions on $\Omega$ and the Orbit Decomposition of $\overline \Omega$}\label{sec:2:7}

In this subsection we shall discuss some basic properties of the triangular subgroups of $G(\Omega)$ which act simply transitively on $\Omega$. This theory actually applies to general homogeneous cones (cf.~\cite{Vinberg}).

From now on, we shall then fix a Euclidean Jordan algebra structure on $F$, with the same scalar product, identity $e_\Omega$, and associated symmetric cone $\Omega$. Since $\Omega$ is assumed to be irreducible, $F$ may be described as in Examples~\ref{ex:1} and~\ref{ex:2}. We shall then fix a frame $(e_1,\dots,e_r)$ of $F$ and simply write $\Delta^{\vect s}$ instead of $\Delta^{\vect s}_{(e_1,\dots,e_r)}$, for every $\vect s\in \C^r$. In addition, we shall write $\Delta_*^{\vect s}$ instead of $\Delta^{\sigma(\vect s)}_{(e_r,\dots,e_1)}$, where 
\[
\sigma(s_1,\dots, s_r)=(s_r,\dots, s_1).
\]
The relevance of this latter functions is partially explained by the following result.

\begin{prop}\label{prop:3b}
	There is $k$ in the stabilizer $K_0$ of $e_\Omega$ in $G_0(\Omega)$ such that $k e_j= e_{r-j-1}$ for every $j=1,\dots,r$. For every such $k$, 
	\[
	\Delta^{\vect s}(x)=\Delta_*^{\sigma(\vect s)}(k x) \qquad\text{and} \qquad \Delta^{\vect s}(x^{-1})=\Delta_*^{-\vect s}(x)=\Delta^{-\sigma(\vect s)}(k^{-1} x)
	\]
	for every $\vect s\in \C^r$ and for every $x\in \Omega$, where $\sigma(s_1,\dots, s_r)=(s_r,\dots, s_1)$.
\end{prop}

\begin{proof}
	The existence of $k$ follows from~\cite[Corollary IV.2.7]{FarautKoranyi}. Notice that such a $k$ is necessarily an automorphism of $F$ as a Euclidean Jordan algebra (cf.~\cite[Theorem III.5.1]{FarautKoranyi}). Then, set 
	\[
	F_j\coloneqq \Set{x\in F\colon (e_1+\cdots+e_j)x=x}\qquad \text{and} \qquad F'_j= \Set{x\in F\colon (e_{r-j+1}+\cdots+e_r) x=x},
	\]
	and let $\pr_j$ and $\pr_{j}'$ be the orthogonal projectors of $F$ onto $F_j$ and onto $F_j'$, respectively.   Then,
	\[
	k F'_j= F_j, \qquad k F_j=F_{j}', \qquad \pr_j k=k\pr_{j}'\qquad \text{and} \qquad \pr_{j}'k=k\pr_j
	\] 
	for every $j=1,\dots,r$. Consequently, $\det_{F'_j}(\pr'_j(k x))=\det_{F'_j}(k \pr_j x)=\det_{F_j}(\pr_j x)$ for every $x\in F$, since $k$ induces an isomorphism of $F_j$ onto  $F'_j$ as Jordan algebras. We have thus proved the first equality. The second equality follows from~\cite[Proposition VII.1.5]{FarautKoranyi}.
\end{proof}

\begin{deff}
	We denote by $\N_\Omega$ the set of $\vect s\in \C^r$ such that $\Delta^{\vect s}$ is polynomial, so that $\N_\Omega=\Set{\vect s\in \N^r\colon s_1\Meg \cdots \Meg s_r}$ (cf.~\cite[Proposition XI.2.1]{FarautKoranyi}). We shall also write $\N^*_\Omega$ instead of $\sigma(\N_\Omega)$, so that $\N^*_\Omega$ is the set of $\vect s\in \C^r$ such that $\Delta_*^{\vect s}$ is polynomial.
\end{deff}
We shall now define, case by case, a group of lower triangular matrices which acts simply transitively on $\Omega$ by linear automorphisms    (cf.~\cite[Proposition VI.3.8]{FarautKoranyi} and also~\cite{Vinberg} for an abstract general construction).

\begin{ex}\label{ex:5}
	Assume that $F$ is the Jordan algebra of Example~\ref{ex:1}. We define $T_-$ as the group of lower triangular $r\times r$ matrices over $\F$ with strictly positive diagonal entries. If $\F\neq \Od$, we let $T_-$ act linearly on $F$  by
	\[
	t\cdot x=t x t^*,
	\]
	so that $T_-$ acts simply transitively on $\Omega$. If $r\meg 3$ and $\F=\Od$, we define the action of $T_-$ on $F$ by describing its differential $\dd\pi$ at the identity, which is a homomorphism of the Lie algebra $T$ of $T_-$, that is, the group of lower triangular $r\times r$ matrices over $\Od$ with real diagonal entries, into $\Lc(F)$. Namely,
	\[
	\dd \pi\colon t\mapsto [F\ni x\mapsto t x+ x t^*\in F].
	\]
	Since $T_-$ is simply connected, the action $t\cdot x$ is well defined, and one may prove that
	\[
	t\cdot e_\Omega= t t^*
	\]
	and that $T_-$ acts simply transitively on $\Omega$.
\end{ex}

\begin{ex}\label{ex:6}
	Assume that $F$ is the Jordan algebra of Example~\ref{ex:2}. We define $T_-$ as the group of formal lower triangular matrices with strictly positive diagonal entries, and we let $T_-$ act linearly on $F$ by
	\[
	\left(\begin{matrix}	a & 0\\ b & c	\end{matrix}\right)\cdot 
	\left(\begin{matrix}	a' & b'\\ b' & c'	\end{matrix}\right)=
	\left(\begin{matrix}	a^2 a' & a c b'+a a' b\\ a c b'+a a' b & c^2 c'+2 c \langle b,b'\rangle+\abs{b}^2 a'	\end{matrix}\right)
	\]
	so that $T_-$ acts simply transitively on $\Omega$ (direct computation). Notice that, in analogy with Example~\ref{ex:5}, one may interpret (formally) $t\cdot x$ as $(t x)t^*=t(x t^*)$, denoting by $t^*$ the transpose of $t$.
\end{ex}

Notice that, in both examples (cf.~\cite[Proposition VI.3.10]{FarautKoranyi}),
\[
\Delta^{\vect s}(t\cdot e_\Omega)=\Delta^{\vect s}_*(e_{\Omega}\cdot t)=\prod_{j=1}^r t_{j,j}^{2 s_j}
\]
for every $t\in T_-$, and for every $\vect s\in \C^r$, denoting by $x\cdot t$ the adjoint of $t\cdot $ evaluated at $x$. We shall therefore also write $\Delta^{\vect s}(t)$ instead of $\prod_{j=1}^r t_{j,j}^{2 s_j}$. Then, the $\Delta^{\vect s}$, $\vect s\in \C^r$, are precisely the characters of $T_-$.

\begin{deff}
	We define $T_-$ and its left action on $F$ as in Examples~\ref{ex:5} and~\ref{ex:6}. We denote by $x\cdot t$ the ajoint action of $t\in T_-$ on $x\in F$. In addition, we define $a=\dim_\R \F$ when $F$ is as in Example~\ref{ex:1} and $a=k$ when $F$ is as in Example~\ref{ex:2}. Then, $a(r-1)/2=m/r-1$.
	
	For every $\epsb\in \Set{0,1}^r$, we define
	\[
	\vect m^{(\epsb)}=\left( a\sum_{k<j} \eps_k  \right)_{j=1,\dots, r} \qquad \text{and} \qquad\vect m'^{(\epsb)}=\left( a \sum_{k>j} \eps_k  \right)_{j=1,\dots, r}
	\]
	and an order relation $\preceq_\epsb$ on $\C^r$ by
	\[
	\vect s \preceq_{\epsb} \vect s' \iff \vect s=\vect s' \lor \vect s'-\vect s\in \epsb(\R_+^*)^r.
	\]
	Hence, $\vect s\prec_{\epsb} \vect s'$ if and only if $s_j<s'_j$ for every $j$ such that $\eps_j=1$, while $s_j=s'_j$ for every $j$ such that $\eps_j=0$ (and $\vect s\neq \vect s'$ if $\epsb=\vect 0$).

	We simply write $\vect m$, $\vect m'$,  $\prec$, and $\succ$ instead of $\vect m^{(\vect 1_r)}$, $\vect m'^{(\vect 1_r)}$,  $\prec_{\vect 1_r}$, and $\succ_{\vect 1_r}$, respectively. 
\end{deff}

\begin{deff}\label{def:1}
We denote by $(I^{\vect s}_\Omega)_{\vect s\in \C^r}$  the unique holomorphic family of tempered distributions on $F$ supported in $\overline \Omega$ such that $\Lc I^{\vect s}=\Delta_{*}^{-\vect s}$ and $\Lc I^{\vect s}_{*}=\Delta^{-\vect s}$ on $\Omega$ for every $\vect s\in \C^r$, where $\Lc$ denotes the Laplace transform on $F$ (cf.~\cite[Proposition 2.28]{CalziPeloso}).  

We define the Gindikin--Wallach sets $\Gc(\Omega)$ and $\Gc_*(\Omega)$ as the sets of $\vect s\in \C^r$ such that $I^{\vect s}$ and $I^{\vect s}_{*}$ are positive Radon measures, respectively, so that $\Gc_*(\Omega)=\sigma(\Gc(\Omega))$.
\end{deff}

Notice that, in particular,  $\Delta^{\vect s}$ and $\Delta^{\vect s}_*$ extend to holomorphic functions on $\Omega+ i F$ for every $\vect s\in \C^r$.

Since we shall sometimes need to consider how the $\Delta^{\vect s}$ interact with the operators $I^{-\vect s'}$, $\vect s'\in \N_{\Omega}^*$, for the  reader's convenience we shall recall the following result (cf.~\cite[Proposition VII.1.6]{FarautKoranyi} or~\cite[Proposition 2.29]{CalziPeloso}).

\begin{lem}\label{lem:8}
Take $\vect s\in \C^r$ and $\vect s'\in \N_{\Omega}^*$. Then,
\[
\Delta^{\vect s}*I^{-\vect s'}= \Big(\vect s+\frac 1 2 \vect m'\Big)_{\vect s'} \Delta^{\vect s-\vect s'}
\]
on $\Omega+ i F$, where $ \big(\vect s+\frac 1 2 \vect m'\big)_{\vect s'}=\prod_{j=1,\dots, r}(s_j+\frac 1 2 m'_j)\cdots  (s_j-s'_j+\frac 1 2 m'_j+1)$.
\end{lem}

In the following result we collect some useful facts about the Gindikin--Wallach sets $\Gc(\Omega)$ and $\Gc_*(\Omega)$ (cf.~\cite{Ishi} for a more detailed treatment).

\begin{prop}\label{prop:17}
The following hold:
\begin{itemize}
	\item[\textnormal{(1)}] $\overline \Omega$ is the disjoint union of the $T_-$-orbits $\Omega_\epsb\coloneqq T_-\cdot e_{\epsb}$ (resp.\ $\Omega_{\epsb}^*\coloneqq e_{\epsb}\cdot T_-$) as $\epsb$ runs through $\Set{0,1}^r$, where $e_{\epsb}=\sum_j \eps_j e_j$;
	
	\item[\textnormal{(2)}]  $\Gc(\Omega)$ is the disjoint union  of the sets of $\vect s\in \R^r$ such that $\vect s \succ_{\epsb} \frac 1 2 \vect m^{(\epsb)}$, as $\epsb$ runs through $\Set{0,1}^r$;
	
	\item[\textnormal{(3)}]  if $\epsb\in \Set{0,1}^r$ and $\Re\vect s\succ \frac 1 2 \vect m^{(\epsb)}$ (resp.\ $\Re\vect s\succ \frac 1 2 \vect m'^{(\epsb)}$), then
	\[
	I^{\vect s}=\frac{1}{\Gamma_{\Omega_\epsb}(\epsb\vect s)} \Delta^{\epsb\vect s}_{\epsb}\cdot \nu_{\Omega_\epsb} \qquad (\text{resp.\ } I^{\vect s}_*=\frac{1}{\Gamma_{\Omega_\epsb^*}(\epsb\vect s)} \Delta^{\epsb\vect s}_{\epsb,*}\cdot \nu_{\Omega_\epsb^*})
	\]
	where $  \Delta^{\vect s'}_{ \epsb}(t\cdot e_{\epsb})=\Delta^{\vect s'}_{\epsb,*}(e_\epsb\cdot t)=\Delta^{\vect s'}(t)  $  for every $t\in T_-$ and for every $\vect s'\in \epsb\C^r$, $\nu_{\Omega_\epsb}$ is a relatively $T_-$-invariant positive Radon measure on $\Omega_\epsb$ with left multiplier $\Delta^{(\vect 1_r-\epsb)\vect m^{(\epsb)}/2}$,  $\nu_{\Omega_\epsb^*}$ is a relatively $T_-$-invariant positive Radon measure on $\Omega_\epsb^*$ with right multiplier $\Delta^{(\vect 1_r-\epsb)\vect m'^{(\epsb)}/2}$, and
	\[
	\Gamma_{\Omega_\epsb}(\epsb\vect s)=\int_{\Omega_\epsb} \Delta^{\epsb\vect s}_{\epsb}(h) \ee^{-\langle e_\Omega,h\rangle}\,\dd \nu_{\Omega_\epsb}(h) \qquad (\text{resp.\ } \Gamma_{\Omega_\epsb^*}(\epsb\vect s)=\int_{\Omega_\epsb^*} \Delta^{\epsb\vect s}_{\epsb,*}(h) \ee^{-\langle e_\Omega,h\rangle}\,\dd \nu_{\Omega_\epsb^*}(h));
	\]
	
	\item[\textnormal{(4)}] if $\vect s\in \epsb \C^r\cap \N_\Omega$ (resp.\ $\vect s\in \epsb \C^r\cap \N_\Omega^*$), then 
	\[
	\Delta^{\vect s}_{\epsb}(h)=\Delta^{\vect s}(h) \qquad (\text{resp.\ } \Delta^{\vect s}_{\epsb,*}(h)=\Delta^{\vect s}_*(h))
	\]
	for every $h\in \Omega_\epsb$ (resp.\ for every $h\in \Omega_\epsb^*$).
\end{itemize}
\end{prop}

The first three assertions follow from~\cite[Theorems 3.5 and 6.2]{Ishi}, while the last one follows observing that, given $h=t\cdot e_{\epsb}\in \Omega_\epsb$, the sequence $h_k\coloneqq [t \sum_{j=1}^r (\eps_j+ (1-\eps_j) 2^{-k} ) e_j]\cdot  e_\Omega $ converges to $h$ in $F$, and $\Delta^{\vect s}(h_k)=\Delta^{\vect s}(t)=\Delta_{ \epsb}^{\vect s}(h)$ for every $k\in\N$. The other half of (4) is proved similarly.

\subsection{Reproducing Kernel Hilbert Spaces of Holomorphic Functions}\label{sec:2:8}
Recall that a reproducing kernel Hilbert space (RKHS for short) of holomorphic functions on $D$ is a vector subspace $H$ of $\Hol(D)$ endowed with the structure of a Hilbert space for which the canonical inclusion $H\subseteq\Hol(D)$ is continuous.
Then, for every $(\zeta,z)\in D$ there is $\Kc_{(\zeta,z)}\in H$ such that 
\[
f(\zeta,z)=\langle f \vert \Kc_{(\zeta,z)}\rangle
\]
for every $f\in H$ and for every $(\zeta,z)\in D$. The sesquiholomorphic function
\[
\Kc\colon ((\zeta,z),(\zeta',z'))\mapsto \Kc_{(\zeta',z')}(\zeta,z)
\]
is called the reproducing kernel of $H$. Observe that the $\Kc_{(\zeta,z)}$, as $(\zeta,z)$ run through $D$, form a total subset of $H$, and that the scalar product of $H$ is therefore completely determined by the relations
\[
\langle \Kc_{(\zeta,z)}\vert \Kc_{(\zeta',z')}\rangle=\Kc((\zeta,z),(\zeta',z'))
\]
for  $(\zeta,z),(\zeta',z')\in D$.

If, conversely, we are given a sesquiholomorphic mapping $\Kc'\colon D\times D\to \C$ such that
\[
\sum_{(\zeta,z),(\zeta',z')\in D} \alpha_{(\zeta,z)}\overline{\alpha_{(\zeta',z')}} \Kc'((\zeta,z),(\zeta',z'))\Meg 0
\]
for every $(\alpha_{(\zeta,z)})\in \C^D$ with finite support, in which case $\Kc'$ is said to be a positive kernel, then we may define a scalar product on the vector space $H'$ generated by the $\Kc'_{(\zeta,z)}=\Kc'(\,\cdot\,,(\zeta,z))$, $(\zeta,z)\in D$, so that
\[
\langle \Kc'_{(\zeta,z)}\vert \Kc'_{(\zeta',z')}\rangle_{H'}=\Kc'((\zeta,z),(\zeta',z'))
\]
for every $(\zeta,z),(\zeta',z')\in D$. Then, $H'$ embeds continuously into $\Hol(D)$ and its completion, canonically identified with a vector subspace of $\Hol(D)$, is a RKHS.

We conclude this subsection observing that, given $H$ and $\Kc$ as above, an automorphism $U$ of $\Hol(D)$ induces a unitary automorphism of $H$ if and only if $(U\otimes \overline U) \Kc=\Kc$ (where $U\otimes \overline U$ or, more precisely, $U\widehat \otimes \overline U$, is defined identifying the space of sesquiholomorphic functions on $D\times D$ with $\Hol(D)\widehat \otimes \overline{\Hol(D)}$).

\subsection{Weighted Bergman Spaces}\label{sec:Bergman}

We now briely review some basic facts on weighted Bergman spaces which are  related to the following discussion. Cf.~\cite{CalziPeloso} for a more thorough discussion of these spaces.

\begin{deff}\label{def:5}
	Take  $\vect s\succ \frac{n+m}{r}\vect 1_r+\frac 1 2 \vect m $. Then, we define 
	\[
	A^{2}_{\vect s}(D)\coloneqq \Set{f\in \Hol(D)\colon \int_{\Omega} \abs{f(\zeta,z)}^2 \Delta^{\vect s-p\vect 1_r}(\Im z-\Phi(\zeta))\,\dd (\zeta,z)<\infty   },
	\]
	endowed with the corresponding norm, where $p=(n+2m )/r$ is the genus of $D$. 
	
	In addition, we define
	\[
	B^{\vect s}_{(\zeta',z')}\colon D\ni (\zeta,z)\mapsto\Delta^{\vect s}\left(\frac{z-\overline{z'}}{2 i}-\Phi(\zeta,\zeta')  \right)\in \C
	\]
	for every $(\zeta',z')\in D$. 
\end{deff}

One may also define corresponding spaces $L^2_{\vect s}(D)$ of measurable functions.

We observe that  $A^{2}_{\vect s}(D)$ is a non-trivial RKHS, and its reproducing kernel is (cf., e.g.,~\cite[Proposition 3.11]{CalziPeloso})
\[
((\zeta,z),(\zeta',z'))\mapsto c_{\vect s} B^{-\vect s }_{(\zeta',z')}(\zeta,z)
\]
for a suitable constant $c_{\vect s}\neq 0$. The case $\vect s=p \vect 1_r$ is of particular importance, since  $A^2_{p\vect 1_r}(D)$ is the unweighted Bergman space, so that its reproducing kernel satisfies remarkable invariance properties. Namely, (cf., e.g.,~\cite[Proposition 1.4.12]{Krantz})
\begin{equation}\label{eq:2b}
 B^{-p\vect 1_r}_{(\zeta',z')}(\zeta,z)=(J\phi)(\zeta,z) \overline{(J\phi)(\zeta',z')} B^{-p\vect 1_r}_{\phi(\zeta',z')}(\phi(\zeta,z))
\end{equation}
for every $(\zeta,z),(\zeta',z')\in D$ and for every $\phi\in G(D)$.

We then denote by $P_{\vect s}$ the  Bergman projector associated with $A^2_{\vect s}(D)$, that is, the orthogonal projector of  $L^2_{\vect s}(D)$ onto $ A^2_{\vect s}(D)$, so that
\[
P_{\vect s} f(\zeta,z)= c_{\vect s}\int_D f(\zeta',z') B^{-\vect s}_{(\zeta',z')}(\zeta,z)\Delta^{ \vect s-p\vect 1_r}(\Im z-\Phi(\zeta))\,\dd (\zeta,z)
\]
for (say) every $f\in C_c(D)$ and for every $(\zeta,z)\in D$.

\begin{deff}
	Take $\vect s\succ\frac{1}{2}\vect m'$. Then, we denote by $\widetilde A^2_{\vect s}(D)$ the unique complete normable space of holomorphic functions on $D$ such that $P_{\vect s'}$ induces a continuous linear mapping of $L^2_{\vect s}(D)$ \emph{onto} $\widetilde A^2_{\vect s}(D)$ for every $\vect s'\succ\frac{n+m}{r}\vect 1_r+ \frac{1}{2}\vect m$ such that $ \vect s'-\frac 1 2 \vect s\succ \frac 1 2\vect m'$ (cf., e.g.,~\cite[Proposition 2.4 and Theorem 4.5]{Paralipomena}).  
\end{deff}

Then, the mapping $f\mapsto f* I^{-\vect s'}$ induces an isomorphism of $\widetilde A^{2}_{\vect s}(D)$ onto $\widetilde A^{2}_{\vect s+\vect s'}(D)$ for every $\vect s'\in \N_\Omega^*$ (cf.~\cite[Proposition 5.13]{CalziPeloso}), and $B^{-\vect s}_{(\zeta,z)}\in \widetilde A^2_{\vect s}(D)$ for every $(\zeta,z)\in D$ (cf.~\cite[Lemma 5.15]{CalziPeloso}). By means of Lemma~\ref{lem:8} we then see that the topology of $\widetilde A^2_{\vect s}(D)$ may be defined by a Hilbert norm with respect to which the reproducing kernel of $\widetilde A^2_{\vect s}(D)$ is $B^{-\vect s}$.

\subsection{Groups of Automorphisms}

\begin{deff}
	We denote by $G(\Omega)$ the group of linear automorphisms of $\Omega$, and by $G_0(\Omega)$ its identity component. 
	
	We denote by $GL(D)$, $\Aff(D)$, and $G(D)$ the groups of linear, affine, and holomorphic automorphisms of $D$, respectively, and by $GL_0(D)$, $\Aff_0(D)$, and $G_0(D)$ their identity components. We simply write $\Aff, \Aff_0, G, G_0$ if there is no fear of confusion.
\end{deff}

Observe that~\cite[p.~14--15]{Kaneyuki} shows that there is a triangular subgroup $T_-'$ of $GL(D)$ which acts simply transitively on $\Omega$. In addition, the canonical mapping $T'_-\ni A\times B_\C\mapsto B\in G(\Omega)$ is an isomorphism onto its image, which is a triangular subgroup of $G(\Omega)$ acting simply transitively on $\Omega$. By~\cite{Vinberg2}, we may then assume that the action of the triangular group $T_-$ constructed in Subsection~\ref{sec:2:7}  induces the group $\Set{B\colon A\times B_\C\in T_-'}$. 
In particular, $T_-$ acts on the left on $E$ in such a way that $t\cdot \Phi(\zeta)=\Phi(t\cdot \zeta)$ for every $t\in T_-$ and $\zeta\in E$. In addition, the semi-direct product $G_T=\Nc \rtimes T_-$ acts simply transitively on $D$ (cf.~\cite[Proposition 2.1]{Murakami}).

\begin{lem}\label{lem:3}
	The group  $G_T$ is solvable, hence amenable. In addition, its characters are the mappings $\Nc \rtimes T_-\ni ((\zeta,x), t)\mapsto \Delta^{\vect s}(t)$, $\vect s\in \C^r$.
\end{lem}

Recall that a group $\Gc$ is said to be amenable if there is a right-invariant mean $\mathtt{m}$ on $\ell^\infty(\Gc)$, that is, a continuous linear functional such that $\mathtt{m}(\chi_\Gc)=1$ and $\mathtt{m}(f(\,\cdot\,g))=\mathtt{m}(f)$ for every $f\in \ell^\infty(\Gc)$. See, e.g.,~\cite{Pier} for more information on amenable groups.

\begin{proof}
	Observe that $G_T$ is solvable since it is the semi-direct product of the nilpotent group $\Nc$ and the solvable group $T_-$.  In particular, $G_T$ is  amenable thanks to~\cite[Corollary 13.5]{Pier}.  Since the $\Delta^{\vect s}$, $\vect s\in \C^r$, are precisely the characters of $T_-$, in order to complete the proof it will suffice to prove that $\Nc\subseteq[G_T,G_T]$. To see that, observe that
	\[
	[(\zeta,x), t]=(\zeta,x)(-t\cdot\zeta,-t\cdot x)=((e-t)\cdot\zeta,(e-t)\cdot x-2 \Im \Phi(\zeta,t\cdot\zeta))
	\]
	for every $(\zeta,x)\in \Nc$. Choosing $ t$ so that $t\cdot (\zeta',z')=(2\zeta',4z')$ for every $(\zeta',z')\in D$, we then see that $\Nc\subseteq [G_T,G_T]$, whence the result.	
\end{proof}

\begin{lem}\label{lem:6}
	Every \emph{positive} character of $\Aff$ (or $\Aff_0$) is uniquely determined by its restriction to $G_T$ (hence to $T_-$). 
	In addition, $\Delta^{\vect s}$ extends to a \emph{positive} character of $\Aff$ (or $\Aff_0$) if and only if $\vect s\in \R \vect d=\R \vect 1_r$, and
	\[
	\Delta^{\lambda \vect 1_r}(\phi)=\abs{{\det}_\C \phi'(0,0)}^{2\lambda/p}
	\]
	for every $\phi\in \Aff$ and for every $\lambda\in \R$, where $p\coloneqq (n+ 2m)/r$ is the genus of $D$.
\end{lem}

\begin{proof}
	Since $G_T$ acts simply transitively on $D$, it is clear that $\Aff=G_T K_\Aff=K_\Aff G_T$, where $K_\Aff$ denotes the stabilizer of $(0, i e_\Omega)$ in $\Aff$. Since $K_\Aff$ is compact (and contained in $GL(D)$, cf.~\cite[Theorem 1.13]{Kaneyuki}), the first assertion follows. 
	
	Next, assume that $\Delta^{\vect s}$ extends to a character of $\Aff_0$.   Then, $\Delta^{\vect s}$ extends to a character of $G_0(\Omega)$ thanks to~\cite[Proposition 4.1 of Chapter V]{Satake}. Consequently, the function $\Delta^{\vect s}$ on $\Omega$ is $K_0$-invariant. Now, by~\cite[Corollary IV.2.7]{FarautKoranyi}, for every permutation $\tau$ of   $\Set{1,\dots,r}$ there is $k_\tau\in K_0$ such that $k_\tau(e_j)=e_{\tau(j)}$ for every $j=1,\dots,r$, so that $\prod_j \alpha_j^{s_j}=\Delta^{\vect s}(\sum_j \alpha_j e_j)= \Delta^{\vect s}(k_\tau\sum_j \alpha_j e_j) =\prod_j \alpha_j^{s_{\tau(j)}}$ for every $\alpha_1,\dots,\alpha_r>0$. Therefore, $\vect s=(s_{\tau(1)},\dots,s_{\tau(r)})$ for every $\tau$, so that $\vect s\in \R\vect 1_r=\R\vect d$. 
	
	Finally, observe that for every $\lambda\in \R$  the mapping
	\[
	\chi\colon \phi\mapsto\abs{{\det}_\C \phi'(0,0)}^{2\lambda/p} 
	\]
	is a well-defined \emph{positive} character of $\Aff$, so that there is $\xi\in \R$ such that $\chi(t\,\cdot\,)=\Delta^{\xi\vect 1_r}(t)$ for every $t\in T_-$, thanks to the previous remarks. Choosing $t=4 e_\Omega$, $\alpha>0$, so that $t\cdot (\zeta,z)=(  2 \zeta, 4 z)$, we then see that $\xi=\lambda$.
\end{proof}

\begin{prop}\label{prop:15}
	The following hold:
	\begin{enumerate}
		\item[\textnormal{(1)}] identifying $T_\Omega=F+i\Omega$ with $\Set{(\zeta,z)\in D\colon \zeta=0}$, the set $G'\coloneqq \Set{g\in G\colon g(T_\Omega)=T_\Omega}$ is a closed subgroup of $G$ and the image of the canonical mapping $G'\to G(T_\Omega)$ contains $G_0(T_\Omega)$;
		
		\item[\textnormal{(1$'$)}]  the set $\Aff'\coloneqq \Set{g\in \Aff\colon g(T_\Omega)=T_\Omega}$ is a closed subgroup of $\Aff$ and the image of the canonical mapping $\Aff'\to \Aff(T_\Omega)$ contains $\Aff_0(T_\Omega)$;
		
		\item[\textnormal{(2)}] there is a $\C$-linear mapping $\phi\colon F_\C\to \Lc(E)$ such that  $\phi(T_\Omega)\subseteq\Aut(E)$, such that
		\[
		\iota\colon D\ni (\zeta,z)\mapsto (-i\phi(z)^{-1}\zeta,-z^{-1})\in D
		\]
		is a well-defined involution of $D$ with $(0,i e_\Omega)$ as its  unique fixed point, and such that $G$ and $G_0$ are generated by $\iota$ and $\Aff$ and $\Aff_0$, respectively;
		
		\item[\textnormal{(3)}]  $\det_\C \iota'(\zeta,z)= i^{-n}\Delta^{-p\vect 1_r}(z)$ for every $(\zeta,z)\in D$, where $p=(n+2m)/r$ is the genus of $D$.
	\end{enumerate}
\end{prop}

\begin{proof}
	It is known that the Lie algebra $\gf$ of $G$ may be endowed with a canonical graduation $(\gf_\lambda)_{\lambda\in \R}$, with $\gf_\lambda=\Set{0}$ if $\lambda\not \in \Set{-1,-1/2,0,1/2,1}$, such that the following hold:
	\begin{itemize}
		\item $\gf_{-1}$ is the Lie algebra of the (closed) subgroup $F\subseteq \Nc$ of $G_0$, acting by translations;
		
		\item $\gf_{-1}\oplus \gf_{-1/2}$ is the Lie algebra of the (closed) subgroup $\Nc$ of $G_0$, acting by translations;
		
		\item $\gf_0$ is the Lie algebra of the (closed) subgroup $GL(D)$ of $G$;
		
		\item $\gf_{-1}\oplus \gf_0\oplus \gf_1$ is the Lie algebra of $G'$.
	\end{itemize}
	See~\cite[Proposition 6.1, Theorem 6.3, Theorem 7.1 and its Corollary]{Murakami} for a proof of the preceding assertions.
	
	(1) By~\cite[Proposition 4.5]{Nakajima2}, $\gf_{-1}\oplus [\gf_{-1},\gf_1]\oplus \gf_1\subseteq \gf_{-1}\oplus \gf_0\oplus \gf_1$ is canonically identified with the Lie algebra of $G(T_\Omega)$. Since the differential of the canonical mapping $\pi\colon G'\to G(T_\Omega)$ is therefore onto, it is clear that the image of $\pi$ is an open subgroup of $G(T_\Omega)$, hence contains $G_0(T_\Omega)$.
	
	(1$'$) The proof is similar to that of (1), since $\gf_{-1}\oplus [\gf_{-1},\gf_1]$ is then canonically identified with the Lie algebra of $\Aff(T_\Omega)$, while $\gf_{-1}\oplus \gf_0$ is canonically identified with the Lie algebra of $\Aff'$. Alternatively, one may apply~\cite[Proposition 4.1 of Chapter V]{Satake}.
	
	(2) The existence of $\phi$ and the fact that $\iota$ is a well-defined involution of $D$ with $(0,i e_\Omega)$ as its unique fixed point follow from~\cite[Corollary 3.6]{Dorfmeister}.
	Then, observe that $\exp_G(\gf_{1/2}\oplus \gf_1)=\iota \Nc \iota$, thanks to~\cite[Theorem 3.9]{Dorfmeister} (observe that $\iota \Nc \iota$ is a connected, simply-connected closed nilpotent subgroup of $G_0$). Then,~\cite[Theorem 6.1]{Dorfmeister} implies that $G=\Nc (\iota \Nc \iota) GL(D) \Nc$, so that $G$ is the group generated by $\Aff$ and $\iota$. In addition, observe that $\iota\in G_0$ (cf.~\cite[Theorem 3.5]{Dorfmeister}), and that $\exp_G(\gf_{-1}\oplus\gf_{-1/2}\oplus\gf_0)\subseteq \Aff_0$ while $\exp_G(\gf_{1/2}\oplus \gf_1)= \iota \Nc \iota$, so that $G_0$ is contained in the group generated by $\Aff_0$ and $\iota$, which is necessarily contained in $G_0$. Then, $G_0$ is generated by $\Aff_0$ and $\iota$.
	
	(3) Observe that there is a constant $c\neq 0$ such that $((\zeta,z),(\zeta',z'))\mapsto c B^{-p\vect 1_r}_{(\zeta',z')}(\zeta,z)$ is the unweighted Bergman kernel (cf., e.g.,~\cite[Proposition 3.11]{CalziPeloso}).
	Setting $J\iota=\det_\C \iota'$, using the invariance properties of the unweighted Bergman kernel~\eqref{eq:2b}, we see that 
	\[
	\begin{split}
		\Delta^{-p\vect 1_r}\Big(\frac{-z^{-1}+i e_\Omega}{2 i}\Big)(J\iota)(\zeta,z) \overline{(J\iota)(0, i e_\Omega)}&=B_{(0, i e_\Omega)}(\iota(\zeta,z)) (J\iota)(\zeta,z) \overline{(J\iota)(0, ie_\Omega)}\\
		&= B_{(0, i e_\Omega)}(\zeta,z)\\
		&=\Delta^{-p\vect 1_r}\Big(\frac{z+i e_\Omega}{2 i}\Big)
	\end{split}
	\]
	for every $(\zeta,z)\in D$. Then, observe that
	\[
	(J\iota)(0, i e_\Omega)=(-1)^n ({\det}_\C\phi(e_\Omega))^{-1} J[z\mapsto -z^{-1}](i e_\Omega)=(-1)^n\Delta^{-2 m/r}(i e_\Omega)=(-1)^{n+m} 
	\]
	by~\cite[p.~341]{FarautKoranyi}, since $\phi(e_\Omega)$ is the identity by~\cite[formula (1.12)]{Dorfmeister}. In addition, observe that
	\[
	\Delta^{-p\vect 1_r}(z_1 z_2)=\Delta^{-p\vect 1_r}(z_1)\Delta^{-p\vect 1_r}(z_2)
	\]
	for every $z_1,z_2\in \C[u]$ and for every $u\in F_\C$  (use~\cite[Proposition II.2.2]{FarautKoranyi}). Then,
	\[
	(J\iota)(\zeta,z)=(-1)^{n+m}\Delta^{-p\vect 1_r}((z+i e_\Omega)  (-z^{-1}+i e_\Omega)^{-1} )=(-1)^{n+m }\Delta^{-p\vect 1_r}(z/i )=(-1)^n i^{n} \Delta^{-p\vect 1_r}(z)
	\]
	for every $(\zeta,z)\in D$, whence the result.
\end{proof}

\subsection{Fourier Analysis on $\Nc$}\label{sec:Fourier}

Since $\Nc$ is a $2$-step nilpotent Lie group (even abelian, if $n=0$), its Fourier transform may be described thoroughly (cf., e.g.,~\cite{OV,Astengoetal} and also~\cite{PWS}).
Here we shall content ourselves with some basic facts which will be useful in the sequel. 

Define 
\[
\Lambda_+\coloneqq \Set{\lambda\in F\colon \forall \zeta\in E\setminus \Set{0}\:\: \langle \lambda, \Phi(\zeta)\rangle>0},
\]
so that $\Lambda_+$ is an open convex cone containing $\Omega$, and its closure is the polar of $\Phi(E)$ in $F$ (cf.~\cite[Proposition 2.5]{PWS}). 
Then, for every $\lambda\in \overline{\Lambda_+}$, there is a unique (up to unitary equivalence) irreducible continuous unitary representation $\pi_\lambda$ of $\Nc$ in some Hilbert space $\Hs_\lambda$  such that $\pi_\lambda(\zeta, x)=\ee^{-i \langle \lambda,x \rangle}$ for every $x\in F$ and for every $\zeta$ in the radical $\Rc_\lambda$ of the positive hermitian form $\langle \lambda, \Phi\rangle$ (cf.~\cite[Subsection 2.3]{PWS}). Notice that we still denote by $\langle\,\cdot\,,\,\cdot\,\rangle$   the $\C$-\emph{bilinear} extension of the scalar product of $F$ to $F_\C$, whereas $\langle \,\cdot\,\vert \,\cdot\,\rangle$ denotes the sesquilinear extension of the scalar product of $F$ to $F_\C$, that is, the scalar product of $F_\C$. In addition, $\Rc_\lambda=\Set{0}$ if (and only if) $\lambda\in \Lambda_+$. 

More explicitly, one may choose $\Hs_\lambda=\Hol(E\ominus\Rc_\lambda)\cap L^2(\nu_\lambda)$, where $E\ominus \Rc_\lambda$ denotes the orthogonal complement of $\Rc_\lambda$ in $E$ and $\nu_\lambda=\ee^{-2\langle \lambda, \Phi\rangle}\cdot \Hc^{2(n-d_\lambda)}$, where $d_\lambda=\dim_\C \Rc_\lambda$ and $\Hc^{2(n-d_\lambda)}$ denotes the $2(n-d_\lambda)$-dimensional Hausdorff measure (i.e., Lebesgue measure), and set
\[
\pi_\lambda(\zeta+\zeta',x) \psi(\omega)\coloneqq \ee^{\langle \lambda, 2 \Phi(\omega,\zeta)  -\Phi(\zeta) -i x \rangle} \psi(\omega-\zeta)
\]
for every $\zeta,\omega\in E\ominus \Rc_\lambda$, for every $\zeta'\in \Rc_\lambda$, for every $x\in F$, and for every $\psi\in \Hs_\lambda$ (cf.~\cite[Subsection 2.3]{PWS}).

Let us now describe the reason why these representations are of particular interest to us. Observe, first, that the orbit $\cM\coloneqq\Nc\cdot (0,0)$ of $(0,0)$ under $\Nc$, which is the \v Silov boundary of $D$, is a CR submanifold of $E\times F_\C$ (cf.~\cite{Boggess} for more information on CR manifolds). In other words, the complex dimension of the `complex' tangent space $T_{(\zeta,z)}\cM\cap i T_{(\zeta,z)}\cM$ of $\cM$ at $(\zeta,z)$, as $(\zeta,z)$ runs through $\cM$, is constant, and equal to $n$.
Observe that the other orbits of $\Nc$ in $E\times F_\C$ are simply translates of $\cM$, so that they all induce the same CR structure on $\Nc$. For this structure, a distribution $u$ on $\Nc$ is CR if and only if $\overline{Z_v} u=0$ for every $v\in E$, where $Z_v$ is the left-invariant vector field on $\Nc$ which induces the Wirtinger derivative $\frac 1 2 (\partial_v -i \partial_{i v})$ at $(0,0)$. In other words,
\[
Z_v=\frac 1 2 (\partial_v -i \partial_{i v})+ i \Phi(v,\,\cdot\,)\partial_F
\]
(cf.~\cite[Subsection 2.2]{PWS}).
If $f\in L^1(\Nc)$ is CR, then $\pi(f)=0$ for every irreducible continuous unitary representation $\pi$ of $\Nc$ which is not unitarily equivalent to one of the $\pi_\lambda$, $\lambda\in \overline{\Lambda_+}$, while $\pi_\lambda(f)=\pi_\lambda(f) P_{\lambda,0}$, where $P_{\lambda,0}$ is the self-adjoint projector  $\Hs_\lambda$ onto the space of constant functions (cf.~the proof of~\cite[Proposition 2.6]{PWS}). If, in addition, there is $g$ in the Hardy space $ H^1(D)$ such that $f=g_h$ for some $h\in \Omega$, where 
\[
g_h\colon \Nc\ni (\zeta,x)\mapsto g((\zeta,x)\cdot (0,i h))=g(\zeta,x+i\Phi(\zeta)+i h),
\] 
then $\pi_\lambda(f)=0$ for every $\lambda\in \overline{\Lambda_+}\setminus \overline{\Omega'}$. Thus, when dealing with CR distributions on $\Nc$ (e.g., the restrictions of holomorphic functions to the translates of $\cM$, or their boundary values if defined), it suffices to consider only the representations $\pi_\lambda$, for $\lambda\in \overline{\Lambda_+}$, or even only for $\lambda\in \overline{\Omega'}$, under some additional assumptions.

We also recall the following useful equality:
\begin{equation}\label{eq:2}
	\tr(\pi_\lambda(\zeta,x) P_{\lambda,0})=\ee^{-\langle \lambda, \Phi(\zeta)+i x\rangle}
\end{equation}
for every $\lambda\in \overline{\Lambda_+}$ and for every $(\zeta,x)\in \Nc$ (cf.~\cite[Proposition 2.3]{PWS}).

Let us now observe, for later use, that if $\lambda\in \overline{\Lambda_+}$ and if $A\in GL(E)$, $B\in GL(F)$ and $A\times B_\C$ is an automorphism of $\Nc$, that is, $B_\C \Phi=\Phi (A\times A)$, then $A\Rc_\lambda=\Rc_{ B^*\lambda}$, and the mapping $\Us_{A, B}\colon \Hs_\lambda\to \Hs_{ B^*\lambda}$ defined by
\[
\Us_{A,B} \psi\coloneqq \abs{{\det}_\C A'} (\psi\circ A'),
\]
where $A'\colon E\ominus \Rc_\lambda\to E\ominus \Rc_{B^*\lambda}$ is the map induced by $A$,\footnote{Notice that the absolute value of the (complex) determinant of a linear map $L$ between two (complex) Hilbert spaces $H_1$ and $H_2$ of the same (finite) dimension is always well defined, and equals the (square root of the) ratio of the (Lebesgue) measures of $L(B_{H_1}(0,1)) $ and $B_{H_1}(0,1)$.}
is unitary, and intertwines $\pi_\lambda \circ (A\times B)$ and $\pi_{ B^*\lambda}$, that is,
\[
\Us_{A,B}\pi_\lambda(A\zeta,B x)=\pi_{B^*\lambda} (\zeta,x)\Us_{A,B}
\]
for every $(\zeta,x)\in \Nc$. In addition, if $A_1\in GL(E)$, $B_1\in GL(F)$ and $A_1\times B_1$ is an automorphism of $\Nc$, then $\Us_{A,B}\Us_{A_1,B_1}=\Us_{A_1 A, B_1 B}$.

We shall then say that a vector field $(v_\lambda)\in \prod_{\lambda\in \overline \Omega} \Hs_\lambda$ is Borel measurable if the mapping 
\[
T_-\ni  t\mapsto \Us_{t, t}^{-1} v_{\lambda\cdot t}\in \Hs_{\lambda}
\]
is Borel measurable for every $\lambda\in \overline{\Omega}$ or, equivalently, for $\lambda=e_{\epsb}$, $\epsb\in \Set{0,1}^r$, thanks to Proposition~\ref{prop:17} and the above discussion. Analogously, we say that a field of operators $(\tau_\lambda)\in \prod_{\lambda\in \overline \Omega} \Lin^2(\Hs_\lambda)$ is Borel measurable if $(\tau_\lambda(v_\lambda))$ is a Borel measurable vector field for every Borel measurable vector field $(v_\lambda)\in\prod_{\lambda\in \overline \Omega} \Hs_\lambda $. If $\tau_\lambda=\tau_\lambda P_{\lambda,0}$ for every $\lambda\in \overline \Omega$, this amounts to saying that the vector field $(\tau_\lambda(e_{\lambda,0}))$ is Borel measurable, where $e_{\lambda,0}$ is the unique \emph{positive} constant function in $\Hs_\lambda$ with norm $1$.

\subsection{Decent and Saturated Spaces}

\begin{deff}\label{def:6}
	Let $X$ be a semi-Hilbert\footnote{That is, a complete prehilbertian space.}  space such that $X\subseteq \Hol(D)$ set-theoretically. 
	We say that $X$ is strongly decent if the set of continuous linear functionals on $X$ which extend to continuous linear functionals on $\Hol(D)$ is dense in the weak dual topology of $X'$.

	We say that $X$ is  saturated if it contains the polar in $\Hol(D)$ of the set of continuous linear functionals on $\Hol(D)$ which induce continuous linear functionals on $X$.
\end{deff}

We recall the following simple result from~\cite[Proposition 2.13]{Rango1}.

\begin{prop}\label{prop:6}
	Let $X$ be a semi-Hilbert space  such that $X\subseteq \Hol(D)$, and let $\Gc$ be a group of automorphisms of $\Hol(D)$ which induce automorphisms of $X$. Then, the following hold:
	\begin{enumerate}
		\item[\textnormal{(1)}] $X$ is strongly  decent if and only if there is a closed $\Gc$-invariant vector subspace $V$ of $\Hol(D)$ such that $X\cap V$ is the closure of $\Set{0}$ in $X$ and the canonical mapping $X\to \Hol(D)/V$ is continuous;
		
		\item[\textnormal{(2)}] $X$ is strongly decent and  saturated if and only if the ($\Gc$-invariant) closure $V$ of $\Set{0}$ in $X$ is closed in $\Hol(D)$ and the canonical mapping $X\to \Hol(D)/V$ is continuous.
	\end{enumerate} 
\end{prop}

Notice that, if $X$ is strongly  decent and $V$ is as in (1), then $X+V$, endowed with the seminorm which is $0$ on $V$ and induces the given seminorm on $X$, is strongly decent and  saturated. In other words, every strongly  decent space has a `saturation'.

\section{Affinely Invariant Spaces of Holomorphic Functions $D$}\label{sec:2}

In this section, we shall first recall some results from~\cite{Ishi3,Ishi4,Ishi5} which characterize the $\vect s\in \C^r$ for which $B^{-\vect s}$ (cf.~Definition~\ref{def:5}) is the reproducing kernel of some RKHS, and then describe and classify the corresponding RKHS according to various kinds of invariance. 
We shall then apply these results in order to study the strongly decent semi-Hilbert spaces of holomorphic functions on the tube domain $F+i \Omega$ in which certain natural representations of $\Aff$ are bounded.

\subsection{The Spaces $\Ac_{\vect s}$}

\begin{deff}
	Take $\vect s\in \R^r$. Then, we define a representation $\Uc_{\vect s}$ of $G_T$ in $\Hol(D)$ setting
	\[
	\Uc_{\vect s}(\phi) f=(f\circ \phi^{-1}) \Delta^{\vect s/2}(\phi^{-1})
	\]
	for every $\phi\in G_T$ and for every $f\in \Hol(D)$. If $\vect s\in \R \vect 1_r$, then we extend $\Uc_{\vect s}$ to $\Aff$ by means of the same formula  (cf.~Lemma~\ref{lem:6}). In other words, we set 
	\[
	\Uc_{\lambda \vect 1_r}(\phi)f=(f\circ \phi^{-1}) \abs{J \phi^{-1}}^{\lambda/p}
	\]
	for every $\lambda\in \R$, for every $\phi\in \Aff$, and for every $f\in \Hol(D)$, where $p=(n+2m)/r$ is the genus of $D$.
\end{deff}

Observe that $[\Uc_{\vect s}(\phi)\otimes \overline{\Uc_{\vect s}(\phi)}] B^{-\vect s}=B^{-\vect s}$ for every $\phi\in G_T$ (cf.~Definition~\ref{def:5}), so that, if $B^{-\vect s}$ is the reproducing kernel of a RKHS $H$, then $H$ is $\Uc_{\vect s}$-invariant with its norm (cf.~Subsection~\ref{sec:2:8}).

We recall the following result, which summarizes some particular cases of~\cite[Theorem A]{Ishi4} and~\cite[Theorem 6]{Ishi6}, where the general case in which $\Omega$ is homogeneous and $\vect s\in \C^r$ is investigated. 

\begin{prop}\label{prop:4b}
	If $\vect s\in \Gc_*(\Omega)$, then $B^{-\vect s}$ is the reproducing kernel of a RKHS $\Ac_{\vect s}$ of holomorphic functions on $D$ in which $U_{\vect s}$ induces an irreducible unitary representation. 
	
	Conversely, if $\vect s\in \R^r$ and $H$ is a RKHS of holomorphic functions on $D$ in which $U_{\vect s}$ induces a bounded (resp.\ unitary) representation, then $\vect s\in \Gc_*(\Omega)$ and $H=\Ac_{\vect s}$ with equivalent (resp.\ proportional) norms.
\end{prop}

We observe explicitly that, in this case, there is virtually no difference between considering the case in which $U_{\vect s}$ induces a bounded or a unitary representation, thanks to the amenability of $G_T$ (cf.~\cite{Kob1,Kob2,Ku,ArazyFisher2}. Namely, since $G_T$ is amenable (cf.~Lemma~\ref{lem:3}), one may always replace the scalar product of $H$ with an equivalent one which is $\Uc_{\vect s}$-invariant, such as
\[
(f,g)\mapsto \mathtt m(\phi\mapsto \langle \Uc_{\vect s}(\phi)f\vert \Uc_{\vect s}(\phi)g\rangle),
\]
where $\mathtt m$ denotes a right-invariant mean on $G_T$.

In particular, $\Ac_{\vect s}=A^2_{\vect s}$ with proportional norms when $\vect s\succ \frac{n+m}{r}\vect 1_r+\frac 1 2 \vect m$, and $\Ac_{\vect s}=\widetilde A^{2}_{\vect s}(D)$ as normable spaces when $\vect s\succ \frac{1}{2}\vect{m'}$ (cf.~Section~\ref{sec:Bergman}).

\begin{deff}
	For every $\vect s\in \Gc_*(\Omega)$, we define $\Ac_{\vect s}$ as the RKHS of holomorphic functions whose reproducing kernel is $B^{-\vect s}$.
	
	In addition, we denote by $\Ms_{\vect s}(\overline \Omega)$ the space of Borel measurable fields of operators $(\tau_\lambda)\in \prod_{\lambda\in \overline \Omega} \Lin^2(\Hs_\lambda)$ such that $\tau_\lambda=\tau_\lambda P_{\lambda,0}$ for every $\lambda\in \overline \Omega$, and such that there is $N\in\N$ such that
	\[
	\int_{\overline \Omega} \frac{\norm{\tau_\lambda}_{\Lin^2(\Hs_\lambda)}}{(1+\abs{\lambda})^N}\,\dd I_*^{\vect s}(\lambda)
	\]
	is finite, modulo the space of $I_*^{\vect s}$-negligible fields of operators (cf.~Subsection~\ref{sec:Fourier}).
	
	We denote by $L^2_{\vect s}(\overline \Omega)$ the space of $(\tau_\lambda)\in \Ms_{\vect s}(\overline \Omega)$ such that
	\[
	\int_{\overline \Omega} \norm{\tau_{\lambda/2}}_{\Lin^2(\Hs_{\lambda/2})}^2\,\dd I_*^{\vect s}(\lambda)=2^{s_1+\cdots+s_r}\int_{\overline \Omega} \norm{\tau_{\lambda}}_{\Lin^2(\Hs_{\lambda})}^2\,\dd I_*^{\vect s}(\lambda)
	\]
	is finite, endowed with the corresponding Hilbert norm.
\end{deff}

Notice that, since $\vect s \Meg \vect 0$ and $I_*^{\vect s}$ is a Radon measure such that $(\rho\,\cdot\,)_* I_*^{\vect s}= \rho^{-(s_1+\cdots+s_r)} I_*^{\vect s}$ for every $\rho>0$,   if a Borel measurable field of operators $(\tau_\lambda)$ is such that $\int_{\overline \Omega} \norm{\tau_\lambda}_{\Lin^2(\Hs_\lambda)}^2\,\dd I_*^{\vect s}(\lambda)$ is finite, then $\int_{\overline \Omega} \frac{\norm{\tau_\lambda}_{\Lin^2(\Hs_\lambda)}}{(1+\abs{\lambda})^N}\,\dd I_*^{\vect s}(\lambda)$ is finite for $N>(s_1+\dots+s_r)/2$. Thus, the definition of $L^2_{\vect s}(\overline \Omega)$ is natural. Further, $L^2_{\vect s}(\overline \Omega)$ is a complete (hence closed) vector subspace of the direct integral $\int_{\overline \Omega}^\oplus \Lin^2(\Hs_{\lambda/2})\,\dd I^{\vect s}_*(\lambda)$.

We are now able to provide a `Fourier-type' description of $\Ac_{\vect s}$ which will be necessary in the sequel.

\begin{prop}\label{prop:1b}
	Take $\vect s\in \Gc_*(\Omega)$, and define $\Pc_{\vect s}\colon \Ms_{\vect s}(\overline \Omega)\to \Hol(D)$ so that
	\[
	\Pc_{\vect s}(\tau)(\zeta,z) \coloneqq \int_{\overline \Omega} \tr(\tau_{\lambda/2} \pi_{\lambda/2}(\zeta,\Re z)^*) \ee^{-\langle \lambda/2, \Im z-\Phi(\zeta)\rangle}\,\dd I_*^{\vect s}(\lambda)
	\]
	for every $\tau=(\tau_\lambda)\in \Ms_{\vect s}(\overline \Omega)$. Then, $\Pc_{\vect s}$ is well defined and one-to-one, and  induces an isometric isomorphism of $L^2_{\vect s}(\overline \Omega)$ onto $\Ac_{\vect s}$.
\end{prop}

\begin{proof}
	The second part is essentially a consequence of~\cite[Proposition 3.6 and Theorem 4.10]{Ishi4}, so that it will suffice to prove that $\Pc_{\vect s}$ is well defined and one-to-one. 
	Observe first that, denoting by $\Lin^1(\Hs_\lambda)$ the space of trace-class endomorphisms of $\Hs_\lambda$,
	\[
	\norm{\tau_\lambda}_{\Lin^1(\Hs_\lambda)}=\norm{\tau_\lambda P_{\lambda,0}}_{\Lin^1(\Hs_\lambda)}\meg \norm{\tau_\lambda}_{\Lin^2(\Hs_\lambda)}
	\]
	for every $\tau\in \Ms_{\vect s}(\overline{\Omega})$ and for  every $\lambda\in \overline{\Omega}$.
	Therefore,
	\[
	\int_{\overline \Omega} \abs{\tr(\tau_{\lambda/2} \pi_{\lambda/2}(\zeta,\Re z)^*) }\ee^{-\langle \lambda/2, \Im z-\Phi(\zeta)\rangle}\,\dd I_*^{\vect s}(\lambda)\meg \int_{\overline \Omega} \norm{\tau_\lambda}_{\Lin^2(\Hs_\lambda)} \ee^{-\langle \lambda/2, \Im z-\Phi(\zeta)\rangle}\,\dd I_*^{\vect s}(\lambda),
	\]
	which is finite for every $(\zeta,z)\in D$ since the function $\lambda \mapsto\ee^{-\langle \lambda/2, \Im z-\Phi(\zeta)\rangle}$ decays exponentially on $\overline \Omega$. In particular, the function $\abs{\tr(\tau_{\lambda/2} \pi_{\lambda/2}(\zeta,\Re z)^*) }\ee^{-\langle \lambda/2, \Im z-\Phi(\zeta)\rangle} $ is uniformly bounded by an $I^{\vect s}_*$-integrable function of $\lambda$ as long as $(\zeta,z)$ stays in a compact subset of $D$. Thus, $\Pc_{\vect s}$ is well defined and maps $\Ms_{\vect s}(\overline \Omega)$ into $C(D)$. Since $\Pc_{\vect s}$ also maps $L^2_{\vect s}(\overline \Omega)$ into $\Ac_{\vect s}\subseteq \Hol(D)$ by the second part of the statement, by approximation we then see that $\Pc_{\vect s}$ maps $\Ms_{\vect s}(\overline \Omega)$ into $\Hol(D)$.
	
	Now, take $\tau\in \Ms_{\vect s}(\overline{\Omega})$ so that $\Pc_{\vect s}(\tau)=0$. Observe that the vector space $V$ generated by the $\ee^{-\langle \,\cdot\,,h\rangle}$, as $h$ runs through $\Omega$, is dense in $C_0(\overline{\Omega})$ by the Stone--Weierstrass theorem. Then, 
	\[
	\tr(\tau_\lambda \pi_\lambda(\zeta,x)^*)=\langle \tau_\lambda e_{\lambda,0}\vert \pi_\lambda(\zeta,x) e_{\lambda,0}\rangle=0
	\]
	for $I^{\vect s}_*$-almost every $\lambda\in \overline{\Omega}$ and for every $(\zeta,x)\in \Nc$, where $e_{\lambda,0}$ is the unique positive constant function with norm $1$ in $\Hs_\lambda$. Since $\pi_\lambda$ is irreducible and $e_{\lambda,0}\neq 0$, this implies that $\tau_\lambda e_{\lambda,0}=0$ for $I^{\vect s}_*$-almost every $\lambda\in \overline{\Omega}$. Since $\tau_\lambda=\tau_\lambda P_{\lambda,0}$ for   every $\lambda\in \overline{\Omega}$, this implies that $\tau=0$, so that $\Pc_{\vect s}$ is one-to-one.
\end{proof}

\begin{prop}\label{prop:2b}
	Take $\vect s,\vect s'\in \Gc_*(\Omega)$. Then, the unitary representations $\Uc_{\vect s}$ and $\Uc_{\vect s'}$ of $G_T$ in $\Ac_{\vect s}$ and in $\Ac_{\vect s'}$, respectively, are unitarily equivalent if and only if there is $\epsb\in \Set{0,1}^r$ such that $\vect s,\vect s'\succ_{\epsb} \frac 1 2 \vect m'^{(\epsb)}$. 
	
	If, in addition, $\vect s'=\vect s+2 \vect s''$ for some $\vect s''\in \N_\Omega^*$, then the mapping $\Ac_{\vect s}\ni  f\mapsto f*I^{-\vect s''}\in \Ac_{\vect s'} $ intertwines $U_{\vect s}$ and $U_{\vect s'}$ and is a multiple of an isometric isomorphism.
\end{prop}

The first assertion is a particular case of~\cite[Theorem 5.3]{Ishi3}. The second assertion follows by means of Propositions~\ref{prop:17} and~\ref{prop:1b} or~\cite[Theorem 4.6]{Ishi5}. In fact, the following elementary lemma holds.

\begin{lem}\label{lem:1}
	Take $\vect s\in \R^r$ and $\vect{s'}\in \N_{\Omega}^*$. Then, for every $f\in \Hol(D)$ and for every $\phi\in G_T$,
	\[
	[\Uc_{\vect s+2 \vect{s'}}(\phi)](f*I^{-\vect{s'}})=(\Uc_{\vect s}(\phi) f)*I^{-\vect{s'}}.
	\]
\end{lem}

\begin{proof}
	The assertion is clear if $\phi\in \Nc$. Then, assume that $\phi=t\cdot$ for $t\in T_-$. Then,
	\[
	(f\circ \phi^{-1})* I^{-\vect{s'}}=[f*(t^* I^{-\vect{s'}})]\circ \phi^{-1}=\Delta^{-\vect{s'}}(t)(f* I^{-\vect{s'}})\circ \phi^{-1},
	\]
	so that the assertion follows. 
\end{proof}

Consequently, by means of Schur's lemma and Proposition~\ref{prop:2b} we get the following result. Notice that this result may also be obtained by means of Lemma~\ref{lem:8} (cf.~Subsection~\ref{sec:2:8}).

\begin{cor}\label{cor:8}
	Take $\epsb\in \Set{0,1}^r$, $\vect s\succ_{\epsb}\frac 1 2 \vect m'^{(\epsb)}$ and $\vect s'\in \N_{\Omega}^*$. Then, the following hold:
	\begin{itemize}
		\item if $\vect s+2 \vect s'\succ_{\epsb} \frac 1 2 \vect m'^{(\epsb)}$ (i.e., if $\vect s'=\epsb\vect s'$), then the mapping $f \mapsto f* I^{-\vect s'}$ is an isomorphism of $\Ac_{\vect s}$ onto $\Ac_{\vect s+2\vect s'}$;
		
		\item if $\vect s+2 \vect s'\not\succ_{\epsb} \frac 1 2 \vect m'^{(\epsb)}$ (i.e., if $\vect s'\neq\epsb\vect s'$), then $\Ac_{\vect s}* I^{-\vect s'}=0$.
	\end{itemize}
\end{cor}

\subsection{Invariant Quotient Spaces}

\begin{deff}
For every $\vect s\in \R^r$ and for every $\vect{s'}\in \N_{\Omega}^*$ such that $\vect s+2 \vect{s'}\in\Gc_*(\Omega)$, we  define
\[
\Ac_{\vect s,\vect{s'}}\coloneqq \Set{f\in \Hol(D)\colon f*I^{-\vect{s'}}\in \Ac_{\vect s+2\vect{s'}}},
\]
endowed with the corresponding Hilbert seminorm. We define   $\widehat\Ac_{\vect s,\vect{s'}}$ as the Hausdorff space associated with $\Ac_{\vect s,\vect{s'}}$, that is, $\Ac_{\vect s,\vect{s'}}/\ker(\,\cdot\, *I^{-\vect{s'}})$.
\end{deff}

As a consequence of Lemma~\ref{lem:1} and~\cite[Theorem 9.4]{Treves}, we have the following result.

\begin{prop}
	Take $\vect s\in \R^r$ and $\vect s'\in \N^*_{\Omega}$ such that $\vect s+2 \vect s'\in \Gc_*(\Omega)$. Then, $\Ac_{\vect s,\vect s'}$ is a semi-Hilbert space, and  $\Uc_{\vect s}$ induces an isometric irreducible representation of $G_T$ in $\Ac_{\vect s,\vect s'}$.
\end{prop}

Notice that the spaces $\widehat \Ac_{\vect s,\vect{s'}}$ for different $\vect{s'}$ need \emph{not} be isomorphic, in general. They are \emph{naturally} isomorphic if (and only if) $\vect s+2\vect{s'}\succ_{\epsb} \frac 1 2 \vect m'^{(\epsb)}$ for some \emph{fixed} $\epsb\in \Set{0,1}^r$, in which case there is a unique isomorphism (up to a scalar multiple) which commutes with $\Uc_{\vect s}$, thanks to Propositions~\ref{prop:4b} and~\ref{prop:2b}.

\begin{prop}\label{cor:1}
	Take $\vect s\in \R^r$ and let $H$ be a   semi-Hilbert space of holomorphic functions on $D$. Assume that the following hold:
	\begin{itemize}
		\item there is $\vect{s'}\in \N_{\Omega}^*$ such that the canonical mapping $H\to \Hol(D)/\ker (\,\cdot\,* I^{-\vect{s'}})$ is continuous and non-trivial;
		
		\item  $\Uc_{\vect s} $ induces a bounded (resp.\ isometric) representation of $G_T$ in $H$.
	\end{itemize} 
	Then, $\vect s+2\vect{s'}\in \Gc_*(\Omega)$, $	H\subseteq \Ac_{\vect s,\vect{s'}}$ continuously, 
	and the canonical mapping $H/(H\cap \ker (\,\cdot\,* I^{-\vect{s'}}))\to \widehat \Ac_{\vect s,\vect{s'}} $ is an isomorphism (resp.\ a multiple of an isometry). 
\end{prop}

Observe that   the canonical mapping $H\to \Hol(D)/\ker (\,\cdot\,* I^{-\vect{s'}})$ is continuous and non-trivial if and only if the mapping $H\ni f \mapsto f *I^{-\vect s'}\in \Hol(D)$ is continuous and non-trivial, since the mapping $f \mapsto f *I^{-\vect s'}$ induces a strict morphism of $\Hol(D)$ \emph{onto} $\Hol(D)$, by the open mapping theorem (use~\cite[Theorem 9.4]{Treves} to prove surjectivity).

\begin{proof}
	This is a consequence of Proposition~\ref{prop:4b} and Lemma~\ref{lem:1}, and of the above remark.
\end{proof}

\subsection{Affinely Invariant Spaces on $D$}\label{sec:3}

We shall now look for $\Aff$-$\Uc_{\lambda\vect 1_r}$-invariant spaces of holomorphic functions.
We shall also consider the following (ray) representations of $G(D)$, which will be the main object of study in the next section.

\begin{deff}\label{def:1b}
	We define, for every $\lambda\in \R$, a representation $\widetilde U_\lambda$ of the universal covering group $\widetilde G$ of $G_0(D)$ so that
	\[
	\widetilde U_\lambda(\phi) f= (f\circ \phi^{-1}) (J\phi^{-1})^{\lambda/p}
	\]
	for every $\phi\in \widetilde G$ and for every $f\in \Hol(D)$, where $p=(n+2m)/r$ is the genus of $D$,  with the conventions described in the Introduction.
	
	We shall also consider the ray representation (cf.~\cite{Bargmann}) $U_\lambda$ of $G(D)$ into $\Lin(\Hol(D))/\T$ defined by
	\[
	U_\lambda(\phi) f= (f\circ \phi^{-1}) (J\phi^{-1})^{\lambda/p}
	\]
	for every $\phi\in G(D)$ and for every $f\in\Hol(D)$.  
\end{deff}

Note that $U_\lambda(\phi)$ may \emph{not} be defined as an ordinary representation of $G(D)$ in $\Hol(D)$ unless $\lambda/p\in \Z$: even though $J\phi^{-1}$ is a nowhere vanishing holomorphic function, so that $(J\phi^{-1})^{\lambda/p}$ may be defined on the \emph{convex} domain $D$, the function $(J\phi^{-1})^{\lambda/p}$ is uniquely defined only up to the multiplication by a power of $\ee^{2 \pi (\lambda/p) i}$. Since, however, these powers are unimodular,  we may still define $U_\lambda$ as a ray representation. In particular, we may say that $U_\lambda$ is bounded or isometric (in a semi-Hilbert space) unambiguously.

In addition, observe that $\abs{U_\lambda(\phi)f}=\abs{\Uc_{\lambda \vect 1_r} (\phi)f}$ for every $\phi\in \Aff$ and for every $f\in \Hol(D)$.

\begin{deff}\label{def:4}
	We denote by $\Wc(\Omega)\coloneqq \Set{\lambda\in \R\colon \lambda \vect 1_r\in \Gc(\Omega')}=\Set{ja/2\colon j=0,\dots, r-1}\cup (m/r-1,+\infty)$ the Wallach set of $\Omega$.

	We shall simply write $\Ac_{\lambda, \lambda'}$ instead of $\Ac_{\lambda \vect 1_r, \lambda' \vect 1_r}$ for every $\lambda\in \R$ and for every $\lambda'\in \N$ such that $\lambda+ 2\lambda'\in \Wc(\Omega)$. We denote by $\widehat \Ac_{\lambda,\lambda'}$ the corresponding Hausdorff space. In addition, we also write $\Ac_\lambda$ instead of $\Ac_{\lambda,0}$.
	We denote by $\square$ the differential operator given by convolution with $I^{-\vect 1_r}_\Omega$, so that $\Ac_{\lambda,\lambda'}=\Set{f\in \Hol(D)\colon \square^{\lambda'} f\in \Ac_{\lambda+2\lambda'}}$ for every $\lambda,\lambda'$ as above.
\end{deff}

We observe explicitly that $\square$ is $K_\Aff$-invariant by Lemma~\ref{lem:6}, where $K_\Aff$ denotes the (compact) stabilizer of $(0, i e_\Omega)$ in $GL(D)$ (or, equivalently, in $\Aff$, cf.~\cite[Theorem 1.13]{Kaneyuki}).

\begin{prop}\label{prop:19}
	Take $\lambda\in \R$ and $\lambda'\in \N$ such that $\lambda+2 \lambda'\in \Wc(\Omega)$. Then, $\Ac_{\lambda,\lambda'}$ is  $\Aff$-$\Uc_{\lambda\vect 1_r}$-invariant with its seminorm.  If, in addition, $\lambda'=0$, then $\Ac_\lambda$ is $U_\lambda$-invariant with its norm.
\end{prop}

Before we pass to the proof, we need a simple extension of Lemma~\ref{lem:1}.

\begin{lem}\label{lem:7}
	Take $\lambda\in \R$ and $\lambda'\in \N$. Then,  for every $f\in \Hol(D)$ and for every $\phi\in \Aff$,
	\[
	[\Uc_{(\lambda+2 \lambda')\vect 1_r}(\phi)](\square^{\lambda'} f)= \square^{\lambda'}(\Uc_{\lambda\vect 1_r}(\phi)f).
	\]
\end{lem}

\begin{proof}
	Observe first that $\Aff=K_\Aff G_T=G_T K_\Aff$. Then, for every $k\in K_\Aff$ and for every $\phi\in G_T$,
	\[
	\begin{split}
		[\Uc_{(\lambda+2 \lambda')\vect 1_r}(k\phi)](\square^{\lambda'} f)&=[\Uc_{(\lambda+2 \lambda')\vect 1_r}(k)\Uc_{(\lambda+2 \lambda')\vect 1_r}(\phi)](\square^{\lambda'} f)\\
			&=[\square^{\lambda'}(\Uc_{\lambda\vect 1_r}(\phi)f)]\circ k^{-1}\\
			&=\square^{\lambda'}(\Uc_{\lambda\vect 1_r}(k\phi)f)
	\end{split}
	\]
	by Lemma~\ref{lem:1} and the $K_\Aff$-invariance of $\square$, which follows from Lemma~\ref{lem:6}.
\end{proof}

\begin{proof}[Proof of Proposition~\ref{prop:19}.]
	The case $\lambda'>0$  follows from the case $\lambda'=0$ and Lemma~\ref{lem:7}. For what concerns the case $\lambda'=0$, observe that $[U_p(\phi)\otimes \overline{U_p(\phi)}] B^{-p \vect 1_r}=B^{-p \vect 1_r}$ for every $\phi\in G(D)$ by~\eqref{eq:2b}. Taking powers, we then see that  $[U_\lambda(\phi)\otimes \overline{U_\lambda(\phi)}] B^{-\lambda \vect 1_r}= c_\phi B^{-\lambda \vect 1_r}$ for some unimodular constant $c_\phi$, and for every $\phi\in G(D)$. Thus, $\Ac_\lambda$ is $U_\lambda$-invariant with its norm. Since $\abs{U_\lambda(\phi)f}=\abs{\Uc_{\lambda \vect 1_r}(\phi)f}$, this also proves that $\Ac_\lambda$ is $\Aff$-$\Uc_{\lambda \vect 1_r}$-invariant with its norm.
\end{proof}

\subsection{The Case of Tube Domains}\label{sec:3:2}

In this subsection, we assume that $D$ is an irreducible symmetric  \emph{tube} domain. Before stating our main results, we need some preliminaries.

Recall that we denote by $G(\Omega)$ the group of linear automorphisms of $\Omega$, and by $G_0(\Omega)$ the component of the identity in $G(\Omega)$. We shall denote by $K$ the stabilizer of $e_\Omega$ in $G(\Omega)$, and by $K_0$ its component of the identity, so that $K_0=K\cap G_0(\Omega)$.

\begin{deff}
Denote by $\Pc_{\vect s}$ the $G_0(\Omega)$-invariant (under composition) subspace of the space of holomorphic polynomials $\Pc$ on $F_\C$ generated by $\Delta^{\vect s}$, for every $\vect s\in \N_\Omega$. 
\end{deff}

\begin{prop}\label{prop:13}
	For every $\vect s\in \N_\Omega$, $\Pc_{\vect s}$ is $G(\Omega)$-invariant. In addition, $\Pc=\bigoplus_{\vect s\in \N_\Omega} \Pc_{\vect s}$ and every $G_0(\Omega)$-invariant  vector subspace of $\Pc$ is the sum of the $\Pc_{\vect s}$ that it contains (and is therefore $G(\Omega)$-invariant).
\end{prop}

\begin{proof}
	The facts that $\Pc=\bigoplus_{\vect s\in \N_\Omega} \Pc_{\vect s}$ and that every $G_0(\Omega)$-invariant vector subspace of $\Pc$ is a sum of the $\Pc_{\vect s}$ follow from~\cite[Theorem XI.2.4]{FarautKoranyi}. It only remains to prove that the $\Pc_{\vect s}$ are $G(\Omega)$-invariant. To this aim, observe first that the $G(\Omega)$-invariant space $\Pc'_{\vect s}$ generated by $\Pc_{\vect s}$ must be a sum of $\Pc_{\vect s'}$ by~\cite[Theorem XI.2.4]{FarautKoranyi}. Now, arguing as in the proofs of~\cite[Lemma XI.2.3 and Theorem XI.2.4]{FarautKoranyi}, one sees that $\Pc'_{\vect s}$ cannot contain $\Delta^{\vect s'}$ unless  $\vect s'=\vect s$, so that $\Pc'_{\vect s}=\Pc_{\vect s}$. Alternatively, one may observe that there is $k\in G(\Omega)$ (possibly in $G_0(\Omega)$) such that $G(\Omega)/G_0(\Omega)=\Set{G_0(\Omega), k G_0(\Omega)}$ and such that $\Delta^{\vect s}\circ k=\Delta^{\vect s}$ for every $\vect s\in \C^r$ (cf.~\cite[p.~42]{Satake}).\footnote{With the notation of Examples~\ref{ex:1} and~\ref{ex:2}, the cases in which $G_0(\Omega)\neq G(\Omega)$ are the following ones: a) $r=2$, in which case one may set $k\left(\begin{smallmatrix} a & b\\ b & c\end{smallmatrix} \right)=\left(\begin{smallmatrix} a & E_{m-2}b\\ E_{m-2} b & c\end{smallmatrix} \right)$, where $E_h=\left(\begin{smallmatrix} -1 &0 \\ 0 & I_{h-1} \end{smallmatrix}\right)$; b) $r\Meg 4$ is even and $\Omega$ is the cone of non-degenerate  positive symmetric real  matrices, in which case one may set $k x=E_{r} x E_r$; c) $r\Meg 3$ and $\Omega$ is the cone of  non-degenerate positive hermitian complex  matrices, in which case one may set $k x = \overline x$.}
\end{proof}

\begin{deff}
Denote by $\widetilde\Dc$ the set of distributions on $F$ supported in $\Set{0}$, and  by $\widetilde\Dc_{\vect s}$ the $G_0(\Omega)$-invariant subspace of $\widetilde\Dc$ generated by $I^{-\vect s}$ for every $\vect s\in\N_{\Omega}^*$. 
\end{deff}

By Proposition~\ref{prop:13} (applied to the $\Delta^{\vect s}_*$ by means of the Laplace transform) we infer that the $\widetilde \Dc_{\vect s}$ are also $G(\Omega)$-invariant, and that $\widetilde \Dc=\bigoplus_{\vect s\in \N_{\Omega}^*} \widetilde\Dc_{\vect s}$.

\begin{prop}\label{prop:9}
	For every $\vect s\in \N_{\Omega}$ and for every $\vect s'\in \N_{\Omega}^*$,
	\[
	\Pc_{\vect s}^\circ=\bigoplus_{\vect s''\neq \sigma(\vect s)}\widetilde\Dc_{\vect s''} \qquad \text{and}\qquad 
	\widetilde\Dc_{\vect s'}=\bigoplus_{\vect s''\neq \sigma(\vect s')} \Pc_{\vect s''}^\circ,
	\]
	where the polars refer to the natural duality between $\Pc$ and $\widetilde\Dc$.
\end{prop}

Recall that  $\sigma(s_1,\dots,s_r)=(s_r,\dots,s_1)$ for every $(s_1,\dots, s_r)\in \C^r$.

\begin{proof}
	Observe that the mapping $\Ic\colon p\mapsto \Fc^{-1}(q(-i\,\cdot\,)) $, where $\Fc^{-1}$ denotes the inverse Fourier transform, induces an isomorphism of   $\Pc$ onto $\widetilde\Dc$, and that for every $q\in \Pc$ and for every $z\in F_\C$
	\[
	\langle \Ic(q),\ee^{\langle \,\cdot\,,z\rangle }\rangle=q( z), \qquad\text{that is,} \qquad \Lc \Ic(q)=q(-\,\cdot\,).
	\]
	Consider the sesquilinear mapping (`Fischer inner product')
	\[
	\langle \,\cdot\,\vert \,\cdot\,\rangle\colon\Pc\times \Pc \ni (q_1,q_2)\mapsto \langle\Ic(q_1),  q_2^*\rangle=\overline{\langle \overline{\Ic(q_1)},q_2\rangle}\in \C
	\]
	where $q_2^*$ is the element of $\Pc$ defined by $q_2^*(z)\coloneqq \overline{q_2(\overline z)}$ for every $z\in F_\C$. Then, $\langle \,\cdot\,\vert \,\cdot\,\rangle$ is a scalar product on $\Pc$ with respect to which the $\Pc_{\vect s}$ are orthogonal to one another (cf.~\cite[Theorem XI.2.4]{FarautKoranyi}). Now, observe that the generators $\Delta^{\vect s}\circ g$, $g\in G_0(\Omega)$, of $\Pc_{\vect s}$ are real on $F$, hence ${}^*$-invariant. Then, $\Pc_{\vect s}$ is ${}^*$-invariant. 
	It will therefore suffice to show that $\Ic(\Pc_{\vect s})=\widetilde\Dc_{\sigma(\vect s)}$ for every $\vect s\in \N_\Omega$. Observe first that, if $q\in \Pc$ and $g\in G_0(\Omega)$, then $\Ic(q\circ g)=(g^*)^*\Ic(q)$, where $(g^*)^*$ denotes the pull-back under the adjoint $g^*$ of $g$ (which still belongs to $G_0(\Omega)$ since $\Omega$ is symmetric).
	Thus, $\Ic(\Pc_{\vect s})$ is the $G_0(\Omega)$-invariant subspace of $\widetilde\Dc$ generated by $\Ic( \Delta^{\vect s})$. Now, by Proposition~\ref{prop:3b}, there is $k\in G_0(\Omega)$ such that
	\[
	(-1)^{s_1+\cdots+s_r}\Lc(\Ic(\Delta^{\vect s}))=\Lc(\Ic(\Delta^{\vect s})(-\,\cdot\,))=\Delta^{\vect s}=\Delta_{*}^{\sigma(\vect s)}\circ k=\Lc(k_* I^{-\sigma(\vect s)})
	\]
	on $\Omega$,	so that  $\Ic(\Delta^{\vect s})=(-1)^{s_1+\cdots+s_r} k_* I^{-\sigma(\vect s)}$. The assertion follows. 
\end{proof}

\begin{deff}
We denote by $\Dc_{\vect s}$, for every $\vect s\in \N_{\Omega}^*$, the space of the continuous linear mappings of the form 
\[
\Hol(D)\ni f \mapsto f * I\in \Hol(D)
\]
as $I$ runs through $\widetilde \Dc_{\vect s}$. We then define $\ker \Dc_{\vect s}$ as $\bigcap_{X\in \Dc_{\vect s}}\ker X$.\footnote{Notice that $\ker \Dc_{\vect s}=\Dc_{\vect s}^\circ$ for the canonical duality between $\Hol(D)$ and the space of differential operators with constant coefficients on $\Hol(D)$.  }
\end{deff}

Notice that a vector subspace of $\Hol(D)$ is $\Aff_0$-$\Uc_{\lambda \vect 1_r}$-invariant if and only if it is $\Aff_0$-$\Uc_{\vect 0}$-invariant, so that we simply say that it is $\Aff_0$-invariant in this case. Similar remarks apply to $\Aff$-invariance.

\begin{cor}\label{cor:6}
	Let $V$ be an $\Aff_0$-invariant closed subspace of $\Hol(D)$. Then, $V$ is $\Aff$-invariant, $V\cap \Pc$ is dense in $V$ and there is $N\subseteq \N_{\Omega}$ such that  $V\cap \Pc=\bigoplus_{\vect s\in N} \Pc_{\vect s}$. In addition, $N'\coloneqq \N_{\Omega}^*\setminus \sigma(N)$ is the set of $\vect s\in \N_{\Omega}^*$ such that $V\subseteq \ker \Dc_{\vect s}$. Finally, $V=\bigcap_{\vect s\in N'} \ker \Dc_{\vect s}$.
\end{cor}

\begin{proof}
	The first assertion follows from~\cite[Proposition 7.1]{Rango1} and Proposition~\ref{prop:13}. Then, take $\vect s\in \N_{\Omega}^*$, and let us prove that $V\subseteq \ker \Dc_{\vect s}$ if and only if $V\cap \Pc\subseteq \ker \widetilde \Dc_{\vect s}$, that is, if and only if $\vect s\in N'$, thanks to Proposition~\ref{prop:9}. Observe first that, if $V\subseteq \ker \Dc_{\vect s}$, then, denoting by $\check I$ the reflection of $I$ (i.e., $(-\,\cdot\,)_*I$),
	\[
	\langle I,q\rangle=(-1)^{s_1+\cdots+s_r}\langle \check I, q\rangle=(-1)^{s_1+\cdots+s_r}(q*I)(0)=0
	\]
	for every $I\in \widetilde \Dc_{\vect s}$ and for every $q\in V\cap \Pc$, thanks to the homogeneity of $I$. Then, $V\cap \Pc \subseteq \ker \widetilde \Dc_{\vect s}$. Conversely, if $V\cap \Pc \subseteq \ker \widetilde \Dc_{\vect s}$, then for every $q\in V\cap \Pc$ and for every $I\in \widetilde \Dc_{\vect s}$, using the translation-invariance of $V$ we see that
	\[
	(q*I)(x)=\langle \check I, q(x+\,\cdot\,)\rangle=(-1)^{s_1+\cdots+s_r} \langle I,q(x+\,\cdot\,)\rangle=0
	\]
	for every $x\in F$, so that $q*I=0$ by holomorphy. By continuity and the arbitrariness of $I$ and $q$, we then infer that $V\subseteq \ker \Dc_{\vect s}$. The last assertion then follows by means of~\cite[Corollary 7.3]{Rango1}.
\end{proof}

\begin{prop}\label{prop:10}
	Take $\vect s,\vect s'\in \N_{\Omega}^*$. Then, $\ker \Dc_{\vect s}\subseteq \ker \Dc_{\vect s+\vect s'}$.
\end{prop}

\begin{proof}
	Take $k\in K_0$ and $f\in \ker \Dc_{\vect s}$. Then, $f*k_* I^{-\vect s}_\Omega=0$, so that $0=f*k_* I^{-\vect s}_\Omega*k_* I^{-\vect s'}_\Omega=f*k_* I^{-\vect s-\vect s'}_\Omega$. By the arbitrariness of $k\in K_0$, this implies that $f\in \ker \Dc_{\vect s+\vect s'}$, whence the result.
\end{proof}

\begin{prop}\label{prop:7}
	Take $\vect s\in \N_\Omega^*$ and $h\in \N$. If  $\vect s \meg h\vect 1_r$, then $\ker \Dc_{\vect s}\subseteq \ker \Dc_{h\vect 1_r}=\ker \square^h$.  
\end{prop}

\begin{proof} 
	By Corollary~\ref{cor:6}, there is $N\subseteq \N_{\Omega}$ such that $\ker \Dc_{\vect s}\cap \Pc=\bigoplus_{\vect s'\in N} \Pc_{\vect s'}$. It  will then suffice to prove that $\square^h \Pc_{\vect s'}=\Set{0}$ for every $\vect s'\in N$. Observe that, since $\square^h$ is $K$-invariant, $\square^h\Pc_{\vect s'}=\Set{0}$  if and only if $\square^h\Delta^{\vect s'}=0$. By Lemma~\ref{lem:8}, this is the case if and only if
	\[
	0=\Big(\vect s'+\frac 1 2 \vect m'\Big)_{h\vect 1_r}=\prod_{j=1}^r \Big(s'_j+\frac 1 2 m'_j\Big)\cdots \Big(s'_j-h+\frac 1 2 m'_j+1\Big)
	\] 
	that is, if and only if there is $j$ such that $s'_j+\frac 1 2 m'_j$ is an integer $<h$. Since $\vect s'$ is decreasing and $m'_r=0$, this is equivalent to saying that $s'_r<h$. Now, if $\vect s'\in N$, then, in particular, $\Delta^{\vect s'} *I^{-\vect s} =0$, so that 
	\[
	0=\Big(\vect s'+\frac 1 2 \vect m'\Big)_{\vect s}=\prod_{j=1}^r \Big(s'_j+\frac 1 2 m'_j\Big)\cdots \Big(s'_j-s_j+\frac 1 2 m'_j+1\Big)
	\] 
	by Lemma~\ref{lem:8} again. Arguing as before, and taking into account the fact that $\vect s$ is increasing, we then see that $s'_r<s_r\meg h$, so that $\square^h \Delta^{\vect s'}=0$. Thus, $\ker \Dc_{\vect s}\subseteq \ker \square^h$. 
\end{proof}

\begin{teo}\label{prop:11}
	Take $\lambda\in \R$.  Let $H$ be a strongly decent non-trivial semi-Hilbert space of holomorphic functions on $D$ such that $\Uc_{\lambda \vect 1_r}$ induces a bounded (resp.\ isometric) representation of $\Aff_0$ in $H$. Then, there are $\ell\in \Set{0,\dots, r}$ and $\vect s\in\N_{\Omega}^*$ such that the following hold:
	\begin{itemize}
		\item   $\lambda \vect 1_r+2\vect s \succ_{\epsb} \frac 1 2 \vect m'^{(\epsb)}$, where $\eps_k=0$ for $k=1,\dots,r-\ell$ and $\eps_k=1$ for $k=r-\ell+1,\dots, r$;
		
		\item $H$ is a dense subspace (resp.\ with a proportional seminorm) of $\Ac_{\lambda ,s_1}+\ker \Dc_{\vect s}$, endowed with the unique seminorm which induces on $\Ac_{\lambda ,s_1}$ its seminorm, and the zero seminorm on $\ker \Dc_{\vect s}$.
		\end{itemize}
\end{teo}

Notice that, if $\ell=r$, then $H$ is a dense subspace of $\Ac_{\lambda,s_r} $, with the above notation, thanks to Proposition~\ref{prop:7}. In addition, all the  spaces described above are clearly (strongly decent, saturated, and) $\Aff$-$\Uc_{\lambda\vect 1_r}$-invariant with their seminorm by Proposition~\ref{prop:19} and Corollary~\ref{cor:6}.

\begin{proof}
	\textsc{Step I.} By Proposition~\ref{prop:6}, there is a closed $\Aff_0$-invariant subspace $V$ of $\Hol(D)$ such that $H\cap V$ is the closure of $\Set{0}$ in $H$ and the canonical mapping $H\to \Hol(D)/V$ is continuous. We may further assume that $V\subseteq H$, that is, that $H$ is saturated.
	Observe that  Corollary~\ref{cor:6} shows that $\Pc\cap V$ is dense in $V$ and that  $V=\bigcap_{\vect s\in N} \ker \Dc_{\vect s}$ for some subset $N$ of $\N_{\Omega}^*$. 
	In particular, for every $\vect s\in N$, the canonical linear mapping $H\to \Hol(D)/\ker \Dc_{\vect s}$ is continuous. Let $N'$ be the set of $\vect s\in N$ such  that this map is non-trivial, that is, such that $H\not \subseteq \ker \Dc_{\vect s}$. Observe that $N'\neq \emptyset$ since the seminorm of $H$ is non-trivial.
	
	Then, take $\vect s\in N'$. Let us first prove that $H\not \subseteq \ker (\,\cdot\,* (k_* I^{-\vect s}))$ for every $k\in K_0$. Indeed, assume by contradiction that $H  \subseteq \ker (\,\cdot\,* (k_* I^{-\vect s}))$ for some $k\in K_0$. Then, for every $k'\in K_0$, 
	\[
	H=H\circ  k k'^{-1}\subseteq \ker (\,\cdot\,* (k_* I^{-\vect s}))\circ k k'^{-1} =\ker (\,\cdot\,* (k'_* I^{-\vect s})),
	\]
	by the $\Aff_0$-invariance of $H$.
	By the arbitrariness of $k\in K_0$, this implies that $H\subseteq \ker \Dc_{\vect s}$, contrary to our choice of $\vect s$. 
	
	In particular, $H\not \subseteq \ker (\,\cdot\,* I^{-\vect s})$, so that Corollary~\ref{cor:1} implies that $\lambda\vect 1_r+2\vect s\in \Gc(\Omega')$, that $H\subseteq \Ac_{\lambda \vect 1_r,\vect s}$ continuously, and that the mapping $H/[H\cap \ker (\,\cdot\,*  I^{-\vect s}_\Omega)]\to \widehat \Ac_{\lambda\vect 1_r,\vect s}$ is an isomorphism (resp.\ a multiple of an isometry).
	By invariance, $H\subseteq  \Ac_{\lambda \vect 1_r,\vect s}^{(k)}\coloneqq \Set{f\in \Hol(D)\colon f\circ k\in \Ac_{\lambda \vect 1_r,\vect s}}$ for every $k\in K_0$, so that $H\subseteq \bigcap_{k\in K_0} \Ac_{\lambda \vect 1_r,\vect s}^{(k)}$. Let us prove that 
	\[
	\bigcap_{k\in K_0} \Ac_{\lambda \vect 1_r,\vect s}^{(k)}=\Ac_{\lambda ,s_r }+\ker \Dc_{\vect s}
	\]
	as vector spaces.
	Observe first that there is $\epsb\in \Set{0,1}^r$ such that $\lambda\vect 1_r+2\vect s\succ_\epsb \frac 1 2 \vect m'^{(\epsb)}$. Since 
	\[
	\vect m'^{(\epsb)}= \left(a \sum_{k>j} \eps_k\right)_j,
	\]
	and since $s_1\meg \cdots\meg s_r$,	this implies that there is $\ell\in \Set{0,\dots, r}$ such that $\eps_k=0$ for $k=1,\dots, r-\ell$ and $\eps_k=1$ for $k=r-\ell+1,\dots, r$. In particular, $\vect m'^{(\epsb)}=(a\min(r-j,\ell) )_j$ and
	\[
	(\lambda+2 s_1)\vect 1_r\succ_\epsb \frac 1 2 \vect m'^{(\epsb)},
	\]
	so that $\lambda+2 s_1\in \Wc(\Omega)$. In addition, setting $\vect s'\coloneqq \vect s- s_1\vect 1_r\in \N_{\Omega}^*$, Proposition~\ref{prop:2b} implies that   the mapping $f\mapsto f* I^{-\vect s'}$  induces a canonical isomorphism from $ \Ac_{\lambda+2 s_1}$ onto $ \Ac_{\lambda \vect 1_r+2\vect s} $ which is a multiple of an isometry and intertwines $\Uc_{(\lambda+2 s_1) \vect 1_r}$ and  $\Uc_{\lambda \vect 1_r + 2 \vect s}$.
	
	Then, take $f\in \bigcap_{k\in K_0} \Ac_{\lambda \vect 1_r,\vect s}^{(k)}$. By the preceding remarks, for every $k\in K_0$ there is a unique $g_k\in \Ac_{\lambda +2s_1}$ such that
	\[
	(f\circ k)*I^{-\vect s}=g_k*I^{-\vect s'},
	\]
	so that  
	\[
	(\square^{s_1}f)\circ k-g_k=\square^{s_1}(f\circ k)-g_k\in \ker(\,\cdot\,* I^{-\vect s'}).
	\]
	Then, for every $k\in K_0$,
	\[
	\square^{s_1}f-g_k\circ k^{-1}\in \ker (\,\cdot\,* k_* I^{-\vect s'}),
	\]
	so that, for every $k,k'\in K_0$,
	\[
	g_k\circ k^{-1}-g_{k'}\circ k'^{-1}\in \ker (\,\cdot\,* k_* I^{-\vect s'})+\ker (\,\cdot\,* k'_* I^{-\vect s'})\subseteq \ker (\,\cdot\,* k_* I^{-\vect s'}* k'_* I^{-\vect s'}).
	\]
	Now, let us prove that $\Ac_{\lambda+2 s_1}\cap \ker (\,\cdot\,* k_* I^{-\vect s'}_\Omega* k'_* I^{-\vect s'}_\Omega)=\Set{0}$. With the notation of Proposition~\ref{prop:1b}, observe that
	\begin{equation}\label{eq:1b}
	\Pc_{(\lambda+2 s_1)\vect 1_r}(\tau)*k_* I^{-\vect s'}* k'_* I^{-\vect s'}= (-1)^{s'_1+\cdots +s'_r}\Pc_{(\lambda+2 s_1)\vect 1_r}(\tau \Delta^{\vect s'}_{*}(k^*\,\cdot\,) \Delta^{\vect s'}_{*}(k'^*\,\cdot\,) )
	\end{equation}
	for every $\tau\in L^2_{(\lambda+2 s_1)\vect 1_r}(\overline{\Omega})$.
	Now, observe that $\Ac^{(k)}_{\lambda+2 s_1}=\Ac_{\lambda+2 s_1}$ by Proposition~\ref{prop:19}, so that the mapping $f\mapsto f* k_* I^{-\vect s'}$ induces an isomorphism of $\Ac_{\lambda+2 s_1}$ onto $\Ac_{\lambda \vect 1_r+\vect s}^{(k)}$, thanks to Proposition~\ref{prop:2b}, applied choosing the Jordan frame $k e_1,\dots, k e_r $ instead of $e_1,\dots,e_r$.
	In particular, $\Ac_{\lambda+2 s_1}\cap \ker (\,\cdot\,* k_* I^{-\vect s'}_\Omega)=\Set{0}$, so that 	Proposition~\ref{prop:1b} shows that $\Delta^{\vect s'}_{*}(k^*\,\cdot\,) $ is non-zero $I^{-(\lambda+2 s_1)\vect 1_r}_{*}$-almost everywhere. Analogously, one proves that $\Delta^{\vect s'}_{*}(k'^*\,\cdot\,) $ is non-zero $I^{-(\lambda+2 s_1)\vect 1_r}_{*}$-almost everywhere. Therefore, $\Delta^{\vect s'}_{*}( k^*\,\cdot\,)  \Delta^{\vect s'}_{*}(k'^*\,\cdot\,)$ is non-zero $I^{-(\lambda+2 s_1)\vect 1_r}_{*}$-almost everywhere. Therefore, Proposition~\ref{prop:1b} and~\eqref{eq:1b} imply that  $\Ac_{\lambda+2 s_1}\cap \ker (\,\cdot\,* k_* I^{-\vect s'}* k'_* I^{-\vect s'})=\Set{0}$. 
	
	We have thus proved that $g_k\circ k^{-1}=g_{k'}\circ k'^{-1} $ for every $k,k'\in K_0$. Call $g$ their common value. Then, $g\in \Ac_{\lambda+2 s_1}$ and
	\[
	\square^{s_1}f -g\in \bigcap_{k\in K_0} \ker (\,\cdot\,* k_* I^{-\vect s'}_\Omega)=\ker \Dc_{\vect s'}.
	\]
	Since $\ker \Dc_{\vect s}=\Set{h\in \Hol(D)\colon \square^{s_1}h \in \ker \Dc_{\vect s'}}$ (cf.~\cite[Theorem 9.4]{Treves}), this implies that
	\[
	f\in \Ac_{\lambda,s_1}+ \ker \Dc_{\vect s}.
	\]
	Conversely, it is clear that $\Ac_{\lambda, s_1}+ \ker \Dc_{\vect s}\subseteq \Ac_{\lambda \vect 1_r, \vect s}^{(k)}$ for every $k\in K_0$.
	We have thus proved that 
	\[
	H\subseteq \Ac_{\lambda,s_1}+ \ker \Dc_{\vect s}
	\]
	continuously, whenever $\vect s\in N'$. 

	\textsc{Step II.} Now, let us prove that, if $\vect s'\in N'$ and $\lambda+2\vect s'\succ_{\epsb'} \frac 1 2 \vect m'^{(\epsb')}$ for some $\epsb'\in \Set{0,1}^r$, then $\epsb'=\epsb$. Indeed, assume by contradiction that $\epsb'\neq \epsb$, and take $\ell'\in \Set{0,\dots, r}$ so that $\eps'_k=0$ for $k=1,\dots , r-\ell'$ and $\eps'_k=1$ for $k=r-\ell'+ 1,\dots, r$.  Up to shifting the roles of $\vect s$ and $\vect s'$, we may assume that $\ell'<\ell$, so that, in particular, $\eps_1'=0$. Then, $s'_1=\frac 1 2 m'^{(\epsb')}_1=\frac 1 2 a\ell'< \frac 1 2 a\ell=s_1$ if $\ell<r$, and $s'_1=\frac 1 2 a\ell'\meg \frac 1 2 a(r-1)<s_1$ if $\ell=r$, so that $s'_1<s_1$ in both cases. Consequently, $s'_{j}=\frac 1 2 a \ell'<s_1\meg s_{j}$ for every $j=1,\dots,r-\ell'$. In addition, take $\vect s''\in (\vect s+\N_\Omega^*)\cap (\vect s'+\N_\Omega^*)$ so that $s''_{j}=s_{j}$ for every $j=1,\dots,r-\ell<r-\ell'$, and observe that  Corollary~\ref{cor:8} and Proposition~\ref{prop:10} imply that
	\[
	H\subseteq \Ac_{\lambda,s'_1}+ \ker \Dc_{\vect s'} \subseteq \ker \Dc_{s_1\vect 1_r}+\ker \Dc_{\vect s'}\subseteq \ker \Dc_{\vect s''}.
	\]
	Now, observe  that $\lambda \vect 1_r+2\vect s'' \succ_{\epsb} \frac 1 2 \vect m'^{(\epsb)}$. Since the canonical mapping $H/(H\cap \ker (\,\cdot\,*I^{-\vect s}))\to \widehat \Ac_{\lambda\vect 1_r,2\vect s}$ is onto, we see that $H*I^{-\vect s}=\Ac_{\lambda\vect 1_r+2\vect s} $. Consequently, Corollary~\ref{cor:8} and Proposition~\ref{prop:10} imply that $\Set{0}=H*I^{-\vect s'' }= \Ac_{\lambda\vect 1_r+2 \vect s}*I^{-(\vect s''-\vect s)}=\Ac_{\lambda\vect 1_r+2\vect s''}$: contradiction.
	
	\textsc{Step III.} Set $\lambda'\coloneqq \min_{\vect s\in N'} s_1$, and observe that the preceding remarks show that $\lambda+2 \lambda'\in \Wc(\Omega)$. More precisely, $\lambda+2 \lambda'>m/r-1$ if $\ell=r$, and $\lambda+2 \lambda'=a\ell/2$ otherwise. Let us prove  that $\Ac_{\lambda,\lambda'}+\ker \Dc_{\vect s}=\Ac_{\lambda,s_1}+\ker \Dc_{\vect s}$ for every $\vect s\in N'$. Indeed, this is obvious if $\ell<r$, in which case $s_1=a \ell/2=\lambda'$ for every $\vect s\in N'$. If, otherwise, $\ell=r$, then the assertion follows from Proposition~\ref{prop:10} and the fact that $\Ac_{\lambda,\lambda'}+\ker  \square^{s_1}=\Ac_{\lambda,s_1}$ as a consequence of   Corollary~\ref{cor:8} and~\cite[Theorem 9.4]{Treves}. 
	
	Let us now prove that $H\cap \ker (\,\cdot\, *I^{-\vect s})\subseteq V$ for every $\vect s\in N'$. To see this, it will suffice to prove that
	\[
	H\cap \ker (\,\cdot\, *I^{-\vect s})\subseteq \ker \Dc_{\vect s'}
	\]
	for every $\vect s'\in N$. If $\vect s'\not \in N'$, then $H\subseteq \ker \Dc_{\vect s'}$, so that $H\cap \ker (\,\cdot\, *I^{-\vect s})\subseteq \ker \Dc_{\vect s'}$. Then,  take $\vect s'\in N'$  and  $f\in H\cap \ker (\,\cdot\,* I^{-\vect s})$. Since $H\subseteq \Ac_{\lambda,s_1'}+\ker \Dc_{\vect s'}=\Ac_{\lambda,\lambda'}+\ker \Dc_{\vect s'}$, there are $f'\in \Ac_{\lambda,\lambda'}$ and $g\in \ker \Dc_{\vect s'}$  such that $f=f'+g$. Then, setting $\vect s''=\vect s+\vect s'-\lambda'\vect 1_r$, and applying Proposition~\ref{prop:10},
	\[
	f'=f-g\in \Ac_{\lambda,\lambda'}\cap (\ker (\,\cdot\, *I^{-\vect s})+\ker (\,\cdot\, *I^{-\vect s'}))\subseteq \Ac_{\lambda,\lambda'} \cap \ker (\,\cdot\,* I^{-\vect s''}).
	\] 
	Since Corollary~\ref{cor:8} shows that the mapping $f\mapsto \ker (\,\cdot\,* I^{-(\vect s''-\lambda'\vect 1_r)})$ induces an isomorphism of $\Ac_{\lambda+2\lambda'}$ onto $\Ac_{\lambda,\vect s''}$, this proves that $f'\in \ker \square^{\lambda'}$.  
	Since $\vect s'-\lambda' \vect 1_r\in \N_{\Omega}^*$ by the definition of $\lambda'$, Proposition~\ref{prop:10} then shows that $f=f'+g\in \ker \Dc_{\vect s'}$. The arbitrariness of $\vect s'$ then shows that $H\cap \ker (\,\cdot\, *I^{-\vect s})\subseteq V$ for every $\vect s\in N'$. Since~\textsc{step I} shows that $V\subseteq \ker (\,\cdot\, *I^{-\vect s})$, this proves that   $H\cap \ker(\,\cdot\,*I^{-\vect s})=V$ for every $\vect s\in N'$.	
	In addition,~\textsc{step I} shows that the canonical mapping $H/V=H/(H\cap \ker (\,\cdot\,*I^{-\vect s}))\to \widehat \Ac_{\lambda\vect 1_r,\vect s}\cong \widehat \Ac_{\lambda,\lambda'}$ is an isomorphism (resp.\ a multiple of an isometry), so that $H\subseteq \Ac_{\lambda,\lambda'}+\ker \Dc_{\vect s}$ with an equivalent (resp.\ proportional) seminorm for every $\vect s\in N'$. 	
\end{proof}

\section{M\"obius-Invariant Spaces on Irreducbile Symmetric Siegel Domains}\label{sec:4}

Recall that we denote by $G$ the group of the biholomorphisms of $D$, and by $G_0$ the identity component of $G$. 
Notice that $G=G_0 \Aff$ (c.f., e.g.,~\cite[Remark 1]{Nakajima}).

In this case, $G_0$ is a simple group, so that none of the representations $\Uc_{\vect s}$ may be extended to $G_0$. We shall therefore only consider the representations $\widetilde U_\lambda$ (and also the ray representations $U_\lambda$), cf.~Definition~\ref{def:1b}.

\begin{oss}\label{oss:1}
	Let us mention that in, e.g.,~\cite{ArazyFisher2,Arazy} some   integrability assumptions were considered instead of our strong decency assumptions. Let us say that a semi-Hilbert subspace $H$ of $\Hol(D)$ satisfies condition $(W\!I)_\lambda$ if: (1) $\widetilde U_\lambda(\phi)$ induces an automorphism of $H$ for every $\phi\in \widetilde G$; (2) $\widetilde U_\lambda$ induces a continuous representation of the stabilizer $\widetilde K$ of $(0,i e_\Omega)$ in $\widetilde G$; (3) the operator $\int_{\widetilde K} \widetilde U_\lambda(\phi)\,\dd \mi(\phi)$, defined as a weak integral with values in $\Lin(\Hol(D))$ endowed with the strong topology, induces an endomorphism of $H$ for every (Radon) measure with compact support in $\widetilde K$; (4) $\langle \int_{\widetilde K} \widetilde U_\lambda(\phi)f\,\dd \mi(\phi) \vert g \rangle_H=\int_{\widetilde K} \langle\widetilde U_\lambda(\phi)f\vert g\rangle_H\,\dd \mi(\phi)$ for every Radon measure $\mi$ with compact support in $\widetilde K$ and for every $f,g\in H$.\footnote{Notice that some of these conditions are stated in a somewhat implicit way in~\cite{ArazyFisher2,Arazy}. Here we added those conditions that do not seem to appear in~\cite{ArazyFisher2,Arazy} but are nonetheless required in the proofs.}
	
	As showed in~\cite[Propositions 2.14 and 6.2]{Rango1} when $r=1$, condition $(W\!I)_\lambda$ holds if and only if $H$ is $\widetilde G$-$\widetilde U_\lambda$-invariant, strongly decent, and saturated. With a similar argument, one may show that condition $(W\!I)_\lambda$ implies that $H$ is strongly decent (and that $H+V$ is strongly decent and saturated, where $V$ is the closure in $\Hol(D)$ of the closure of $\Set{0}$ in $H$), and that if $H$ is $\widetilde G$-$\widetilde U_\lambda$-invariant, strongly decent, and saturated, then condition $(W\!I)_\lambda$ holds.
	
	Since, however, the proof of~\cite[Theorems 5.2]{Arazy} seems to be incomplete under the sole assumption $(W\!I)_\lambda$ (unless $r=1$ or a saturation assumption is added), there appears to be no loss of generality if we consider strongly decent and saturated spaces only.
\end{oss}

The following result is essentially a particular case of~\cite[Theorem 3]{ArazyFisher2}. It may also be seen as a consequence of Propositions~\ref{prop:4b} and~\ref{prop:19}. Cf.~also~\cite{Ishi6} for the case in which $D$ is a bounded homogeneous domain.

\begin{prop}\label{prop:20}
	Take $\lambda\in \R$. If $\lambda\in \Wc(\Omega)$, then $\Ac_{\lambda}$ is $G$-$U_\lambda$-invariant with its norm.
	
	Conversely, if $H$ is a non-trivial Hilbert space which is continuously embedded in $\Hol(D)$ and in which  $U_\lambda$ induces a bounded (resp.\ isometric) representation of $G_T$, then $\lambda\in \Wc(\Omega)$ and $H=\Ac_\lambda$ with equivalent (resp.\ proportional) seminorms.
\end{prop}

\subsection{The Case of Tube Domains}

In this section we extend~\cite{Garrigos}. Notice that the fact that $\Ac_\lambda$ is $G$-$U_\lambda$-invariant for $\lambda\in \Wc(\Omega)$ is contained in Proposition~\ref{prop:20}. 

\begin{teo}\label{teo:2}\label{teo:4}
	Take $\lambda\in \R$. 
	If $\lambda\in m/r-1-\N$, then $\Ac_{\lambda,m/r-\lambda}$ is $G$-$U_\lambda$-invariant with its seminorm.
	
	Conversely, let $H$ be a non-trivial strongly decent and saturated semi-Hilbert space of holomorphic functions on $D$ in which $U_\lambda$ induces a bounded (resp.\ isometric) ray representation of $G_0$. Then, either one of the following hold:
	\begin{itemize}
		\item $\lambda\in \Wc(\Omega)$ and $H=\Ac_{\lambda }$ with equivalent (resp.\ proportional) norms;
		
		\item $\lambda\in m/ r-1-\N$ and $H=\Ac_{\lambda, m /r-\lambda}$ with   equivalent (resp.\ proportional) seminorms. 
	\end{itemize}
\end{teo}

This result partially extends~\cite{Garrigos} to the case $\lambda\neq 0$. This result also extends~\cite[Theorem 5.2]{Arazy} for the case of tube domains, because of Remark~\ref{oss:1}. Notice that we do not assume that the $U_\lambda(\phi)$ are isometries on $H$.

In order to prove the main result of this section, we need two propositions, which are both interesting in their own right. 
The first one shows that $U_\lambda$ and $U_{2 m/r-\lambda}$  are intertwined (up to a unimodular constant) by $\square^{m/r-\lambda}$ when $\lambda\in m/r-1-\N$, and is a consequence of~\cite[Theorem 6.4]{Arazy}. 
As we shall see later, the analogous assertion does \emph{not} hold when $n>0$.

The second one characterizes the closed $G_0$-$U_\lambda$-invariant subspaces of $\Hol(D)$.

\begin{prop}\label{prop:4}
	Take $\lambda\in m/r-1-\N$. Then, 
	\[
	U_{2m/r-\lambda}(\phi)\square^{m/r-\lambda} f=\square^{m/r-\lambda} U_\lambda(\phi) f
	\]
	for every $\phi\in G(D)$ and for every $f\in \Hol(D)$ (equality  in $\Hol(D)/\T$).  
\end{prop}

Notice that this implies that $\square^{m/r-\lambda}$ intertwines $\widetilde U_{2 m/r-\lambda}$ and $\widetilde U_{\lambda}$   as (ordinary) representations of $\widetilde G$ into $\Hol(D)$ (cf.~\cite[Theorem 3.2]{Bargmann}). 
As observed in~\cite{Garrigos} for the case $\lambda=0$, deriving this result by means of~\cite[Theorem 6.4]{Arazy} (which is its analogue for a circular bounded realization of $D$) is not straightforward. Following~\cite{Garrigos}, we shall therefore provide a direct proof.

\begin{proof}
	Observe first that the assertion follows from Lemma~\ref{lem:7} when $\phi\in \Aff$, and that $G$ is generated by $\Aff$ and the inversion $\iota\colon z \mapsto -z^{-1}$ (cf.~Proposition~\ref{prop:15}). Since  $U_\lambda$ and $U_{2m/r-\lambda}$ are ray representations of $G$, it will then suffice to prove our assertion for $\phi=\iota$.
	Observe first that, by Proposition~\ref{prop:15}, $J\iota =\Delta^{-(2 m/r) \vect 1_r} = (2i)^{-2m} B_0^{-(2m/r)\vect 1_r}$,\footnote{Slightly abusing the notation, we write $B_0^{\vect s}(z)$ instead of $\Delta^{\vect s}(z/(2i))$ for every $z\in D$ and for every $\vect s\in \C^r$.} so that  $(J \iota)^\xi\coloneqq 2^{-2 m \xi} B_0^{-(2 \xi m/r) \vect 1_r}$ (up to a unimodular constant) on $D$, for every $\xi\in \R$. In particular, it will suffice to prove that
	\[
	\square^{m/r-\lambda} [(f\circ \iota)B_0^{- \lambda  \vect 1_r} ]=4^{r\lambda- m} B_0^{- (2 m/r-\lambda) \vect 1_r}  (\square^{m/r-\lambda} f)\circ \iota
	\]
	for every $f\in \Hol(D)$ (equality in $\Hol(D)$). By the proof of~\cite[Lemma 3.8]{Garrigos}, we see that
	\[
	\square^{m/r-\lambda} [(q\circ \iota)B_0^{- \lambda  \vect 1_r} ]=(-2 i)^{r\lambda-m} 	\Big( \vect s+\Big(\frac m r -1\Big)\vect 1_r -\frac 1 2 \vect m\Big)_{(m/r-\lambda)\vect 1_r} (q\circ \iota) B_0^{-(m/r)\vect 1_r}
	\]
	on $D$ for every $\vect s\in \N_\Omega$ and for every $q\in \Pc_{\vect s}$, where $\Pc_{\vect s}$ is the $G_0(\Omega)$-invariant vector space generated by $\Delta^{\vect s}$ (cf.~Subsection~\ref{sec:3:2}). Now, by~\cite[Lemma XIV.2.1]{FarautKoranyi},
	\[
	\Delta^{(m/r-\lambda)\vect 1_r} \square^{m/r-\lambda} q=	\Big( \vect s+\Big(\frac m r -1\Big)\vect 1_r -\frac 1 2 \vect m\Big)_{(m/r-\lambda)\vect 1_r} q
	\]
	for every $\vect s\in \N_\Omega$ and for every $q\in \Pc_{\vect s}$. In addition, Proposition~\ref{prop:3b} shows that $\Delta^{(m/r-\lambda)\vect 1_r}\circ\iota=(2 i)^{r\lambda-m} B^{(\lambda-m/r)\vect 1_r}_0$.
	Since $\bigoplus_{\vect s\in \N_\Omega} \Pc_{\vect s}$ is the space of holomorphic poynomials  on $F_\C$ by Proposition~\ref{prop:13}, this proves that
	\[
	\square^{m/r-\lambda} [(f\circ \iota)B_0^{- \lambda  \vect 1_r} ]= 4^{r\lambda-m}B_0^{- (2 m/r-\lambda) \vect 1_r}  (\square^{m/r-\lambda} f)\circ \iota
	\]
	for every holomorphic polynomial $f$, hence for every $f\in \Hol(D)$, since the space of holomorphic polynomials is dense in $\Hol(D)$ by~\cite[Corollary 7.2]{Rango1}.
\end{proof}

\begin{prop}\label{prop:12}
	Take $\lambda\in \R$ and a closed vector subspace $V$ of $\Hol(D)$.  Let $K_\lambda$ be the set of $k\in\Set{1,\dots,r}$ such that  $ \frac 1 2 m_{k}-\lambda=\frac{a(k-1) }{2}-\lambda\in\N$. 
	For every $k\in K_\lambda$, define 
	\[
	N_{\lambda,k}\coloneqq\Set{\vect s\in \N^*_\Omega\colon	s_{ r-k+1}=\cdots=s_{ r}=\frac 1 2 m_k-\lambda+1 }.
	\]
	Then, $V$ is $G_0$-$U_\lambda$-invariant if and only if it is either $\Set{0}$, $\Hol(D)$, or $V_{\lambda,k}\coloneqq \bigcap_{\vect s\in N_{\lambda,k}}\ker \Dc_{\vect s}$ for some $k\in K_\lambda$.
	If this is the case, then $V$ is also $G$-$U_\lambda$-invariant.
	
	Finally, if $k\in K_\lambda$ and $\vect s\in N_{\lambda,k}$, then $V_{\lambda,k} $ is the largest $G_0$-$U_\lambda$-invariant closed vector subspace of $\ker \Dc_{\vect s}$, and
	\[
	V_{\lambda,k}=\Set{f\in \Hol(D)\colon \forall \phi\in G_0\:\: (U_\lambda(\phi) f)*I^{-\vect s}=0  }.
	\]
\end{prop}

Observe that this essentially provides (cf.~Subsection~\ref{sec:bounded}) a particular case of~\cite[Theorem 4.8, (ii)]{Arazy}.
Observe, in addition, that the sets $N_{\lambda,k}$ are finite, since the elements of $\N^*_\Omega$ are increasing.

\begin{proof}
	\textsc{Step I.} Set $a_{k}\coloneqq -\lambda+\frac 1 2 m_k$ for every $k=1,\dots, r$, so that $K_\lambda= \Set{k\in \Set{1,\dots, r}\colon a_{k}\in \N}$. Set $q(\lambda)\coloneqq \card(K_\lambda)$ and let $k_1,\dots,k_{q(\lambda)}$ be the elements of $K_\lambda$, ordered increasingly.
	
	Assume that $V$ is $G_0$-$U_\lambda$-invariant and  that $V\neq \Set{0}, \Hol(D)$. Since, in particular, $V$ is $\Aff_0$-invariant, Corollary~\ref{cor:6} implies that   there is a subset $N$ of $\N_{\Omega}$ such that $V$ is the closure of $\bigoplus_{\vect s\in N}  \Pc_{\vect s}$, so  that $V=\bigcap_{\vect s\in  \N_{\Omega}\setminus N} \ker \Dc_{\sigma(\vect s)}$, and  $V\subseteq \ker \Dc_{\sigma(\vect s)}$ if and only if $\vect s\in \N_{\Omega}\setminus N$,  where $\sigma(s_1,\dots,s_r)=(s_r,\dots,s_1)$.
	In particular, $N\neq \emptyset, \N_\Omega$. 
	Now, define $\iota\colon z \mapsto -z^{-1}$, so that  $U_\lambda(\iota) f= (f\circ \iota) \Delta^{-\lambda \vect 1_r} =2^{-r\lambda}(f\circ \iota) B_0^{-\lambda \vect 1_r}$ for every $f\in \Hol(D)$, with equality in $\Hol(D)/\T$ (cf.~the proof of Proposition~\ref{prop:4}).  
	
	Take $\vect s\in N$ and observe that  there is $k$ in the stabilizer of  $e_\Omega$ in $G_0(\Omega)$ (canonically identified with the  stabilizer of $i e_\Omega$ in $GL_0(D)$) such that (cf.~Proposition~\ref{prop:3b} and Lemma~\ref{lem:6})
	\[
	U_\lambda(\iota) B^{\vect s}_{0}= 2^{-r \lambda}(B^{-\sigma(\vect s)}_0 \circ k) B_0^{-\lambda \vect 1_r}=  2^{-r \lambda} B^{-\sigma(\vect s)-\lambda\vect 1_r}_0 \circ k.
	\] 
	Now, take $\vect s'\in \N_\Omega \setminus N$. Since $(U_\lambda(\iota) B^{\vect s}_{0})* k^{-1}_* I^{-\sigma(\vect s')} =0$,  Lemma~\ref{lem:8} shows that
	\[
	0=B^{-\sigma(\vect s)-\lambda\vect 1_r}_0*I^{-\sigma(\vect s')}=(2i)^{-(s'_1+\cdots+s'_r)} \left(-\sigma(\vect s)-\lambda\vect 1_r+\frac 1 2 \vect m'\right)_{\sigma(\vect s')} B^{-\sigma(\vect s)-\lambda\vect 1_r-\sigma(\vect s')}_0,
	\]
	so that 
	\[
	\prod_{k=1}^r \left(-s_{r-k+1}-\lambda+ \frac 1 2 m'_{k}\right)\cdots \left(-s_{r-k+1}-\lambda - s'_{r-k+1}+ \frac 1 2 m'_{k}+1\right)= \left(-\sigma(\vect s)-\lambda\vect 1_r+\frac 1 2 \vect m'\right)_{\sigma(\vect s')}=0 .
	\]
	In other words, noting that $\sigma(\vect m')=\vect m$, there is $k\in K_\lambda$ such that
	\[
	a_k\Meg s_k>a_k-s'_k.
	\]
	Observe that $a_k, -s_k$, and $ a_k-s'_k$ are increasing functions of $k$.
	
	Define, for every $j=1,\dots, q(\lambda)$, $N_j\coloneqq \Set{ \vect s''\in \N_{\Omega}\colon s''_{k_j}\meg a_{k_j}}$, so that $N_1\subseteq\cdots \subseteq N_{q(\lambda)}$. Observe that, if $\vect s\in N$, then $s_{k_j}\meg a_{k_j}$ for some $j\in \Set{1,\dots, q(\lambda)}$ by the previous remarks (since $N\neq \N_\Omega$), so that $\vect s\in N_j\subseteq N_{q(\lambda)}$. Thus, $N\subseteq N_{q(\lambda)}$.
	
	Now, let $\bar \jmath$ be the smallest $j\in\Set{1,\dots,q(\lambda)}$ such that $N\subseteq N_j$, and let us prove that $N=N_{\bar \jmath}$. Indeed, assume on the contrary that there is $\vect s'\in N_{\bar \jmath}\setminus N$, so that $s'_{k_{\bar\jmath}}\meg a_{k_{\bar \jmath}}$. 
	Take $\vect {\bar s'} \in \N_\Omega$ so that $\bar s'_1=\cdots= \bar s'_{k_{\bar \jmath-1}}=a_{k_{\bar\jmath-1}}+1$ while $\bar s'_{k_{\bar\jmath}}=\cdots=\bar s'_{r}=0$ (we do not impose any conditions on the possibly remaining $\bar s'_k$). Then, for every $j=1,\dots, \bar\jmath-1$,
	\[
	\bar s'_{k_j}=a_{k_{\bar\jmath-1}}+1>a_{k_j},
	\]
	whereas, for $j=\bar\jmath ,\dots, q(\lambda)$,
	\[
	\bar s'_{k_j}= 0\meg a_{k_{\bar\jmath}} -s'_{k_{\bar\jmath}} \meg a_{k_j}-s'_{k_j},
	\]
	so that $\vect {\bar s'}\not \in N$ by the previous remarks. Hence, for every $\vect s\in N$ there is $j\in \Set{1,\dots,q(\lambda)}$ such that
	\[
	a_{k_j}\Meg s_j> a_{k_j}-\bar s'_j,
	\]
	so that necessarily $\bar s'_j>0$, whence $\bar \jmath\Meg 2$ (since $N\neq \emptyset$) and $j\meg \bar \jmath-1$. We have thus proved that $N\subseteq N_{\bar\jmath-1}$, contrary to the definition of $\bar\jmath$. It then follows that $N=N_{\bar\jmath}$.
	
	Let us then show that $V=\bigcap_{\vect s\in N_{\lambda,k_{\bar \jmath}}}\ker \Dc_{\vect s}$. 
	To this aim, observe that $V= \bigcap_{\vect s\in N'_{\bar \jmath}} \ker \Dc_{\vect s}$, where $N'_{\bar \jmath}\coloneqq \N_{\Omega}^*\setminus \sigma(N_{\bar \jmath})=\Set{\vect s\in \N_{\Omega}^*\colon s_{r-k_{\bar \jmath}+1}\Meg a_{k_{\bar \jmath}}+1}$. 
	Observe that, if $\vect s\in N'_{\bar \jmath}$, then $\vect s=\vect s'+\vect s''$, where $s'_j=\min(s_j, a_{k_{\bar \jmath}}+1)$ and $s''_j=(s_j- a_{k_{\bar \jmath}}-1)_+$ for every $j=1,\dots,r$. Then, $\vect s'\in N_{\lambda,k_{\bar \jmath}}$ and $\vect s''\in \N_\Omega^*$, so that Proposition~\ref{prop:10} implies that  $\ker \Dc_{\vect s'} \subseteq \ker \Dc_{\vect s}$. Since clearly $N_{\lambda,k_{\bar \jmath}}\subseteq N'_{\bar \jmath} $, this proves that $V=\bigcap_{\vect s\in N_{\lambda,k_{\bar \jmath}}}\ker \Dc_{\vect s}$. Since $G=\Aff G_0$ by~\cite[Remark 1]{Nakajima}, this implies that $V$ is actually $G$-$U_\lambda$-invariant.
	
	\textsc{Step II.} Now, observe that~\cite[Theorem 5.3]{FarautKoranyi2} shows that there  are at least $q(\lambda)$  closed $G_0$-$U_\lambda$-invariant subspaces of $\Hol(D)$ which  are different from $\Set{0}$ and $\Hol(D)$.  By~\textsc{Step I}, these spaces must be the $V_{\lambda,k}$, $k\in K_\lambda$. In particular, the $V_{\lambda,k}$, $k\in K_\lambda$, are $G_0$-$U_\lambda$-invariant.
	
	\textsc{Step III.} Now, take $k\in K_\lambda$ and $\vect s\in N_{\lambda,k}$. Observe that~\textsc{step I} shows that, if $V$ is a $G_0$-$U_\lambda$-invariant closed vector subspace of $\ker \Dc_{\vect s}$, then it is of the form $V_{\lambda,k'}$ for some $k'\in K_\lambda$. In particular, $\vect s\in N'_{k'}$, that is, $s_{r-k'+1}\Meg a_{k'}+1$. Thus, $a_k+1\Meg a_{k'}+1$, so that $k\Meg k'$ and  $V=V_{\lambda,k'}\subseteq V_{\lambda,k}$. Thus, $V_{\lambda,k}$ is the largest $G_0$-$U_\lambda$-invariant closed vector subspace of $\ker \Dc_{\vect s}$. Since the same holds replacing $\ker \Dc_{\vect s}$ with $\ker (\,\cdot\,*I^{-\vect s})$ (as $G_0$ contains $GL_0(D)\cong G_0(\Omega)$), this proves that
	\[
	V_{\lambda,k}=\Set{f\in \Hol(D)\colon \forall \phi\in G_0\:\: (U_\lambda(\phi) f)*I^{-\vect s}=0  },
	\]
	whence the conclusion.
\end{proof}

\begin{proof}[Proof of Theorem~\ref{teo:2}.]
	The first assertion follows immediately from Proposition~\ref{prop:4} and the $G$-$U_{2 m/r-\lambda}$-invariance of $\Ac_{2m /r-\lambda}$ (cf.~Proposition~\ref{prop:20}).

	Then, consider the second assertion.
	Denote by $V$ the closure of $\Set{0}$ in $H$, so that $V$ is a \emph{proper} closed $G_0$-$U_\lambda$-invariant vector subspace of  $\Hol(D)$ and the linear mapping $H\to \Hol(D)/V$ is continuous. By Proposition~\ref{prop:12}, we see that either $V=\Set{0}$, in which case Proposition~\ref{prop:4b} leads to the conclusion, or there is $k\in \Set{1,\dots,r}$ such that $\frac1 2 m_{k}-\lambda\in \N$ and $V=\bigcap_{\vect s\in N_{\lambda,k}}\ker \Dc_{\vect s}$, with the notation of Proposition~\ref{prop:12}. 
	Let us show that $k=r$. Assume by contradiction that $k<r$, and observe that, by Theorem~\ref{prop:11}, there is $\ell\in \Set{0,\dots,r}$ such that $H$ is a dense vector subspace of $\Ac_{\lambda,\lambda'}+V$ (resp.\ with proportional seminorms) for some $\lambda'\in \N$ such that $(\lambda+2 \lambda')\vect 1_r\succ_\epsb \frac 1 2 \vect m'^{(\epsb)}$, where $\epsb\in \Set{0,1}^r$ is defined by $\eps_1=\dots=\eps_{r-\ell}=0$ and $\eps_{r-\ell+1}=\cdots=\eps_{r}=1$. 
	Let $\vect s$ be the minimum of $N_{\lambda,k}$, so that  $s_1=\cdots=s_{r-k}=0$. 
	Then,~\textsc{step II} in the proof of Theorem~\ref{prop:11} shows that either  $H\subseteq \ker \Dc_{\vect s}$, or $\lambda\vect 1_r + 2 \vect s\succ_\epsb \frac 1 2 \vect m'^{(\epsb)}$. 
	The first case cannot occur, since it would imply that $H\subseteq V$ by the $G_0$-$U_\lambda$-invariance of $H$ and Proposition~\ref{prop:12}, and this would contradict the assmption that $H$ be non-trivial.
	Then, $\lambda\vect 1_r + 2 \vect s\succ_\epsb \frac 1 2 \vect m'^{(\epsb)}$. Since $s_1=0$ and $\vect m'^{(\epsb)}$ is decresing, this implies that $\lambda \vect 1_r\succ_{\epsb} \frac 1 2 \vect m'^{(\epsb)}$. If $\ell=r$, this implies that $\lambda > \frac 1 2 m'_1$, so that $\frac{1}{2} m_k-\lambda\meg \frac 1 2 m_r -\lambda=\frac 1 2 m'_1-\lambda<0$, which is absurd.
	Then, $\ell<r$ and   $\lambda=\frac1 2 m'^{(\epsb)}_{r-\ell}=\frac 1 2 m'_{r-\ell}$.
	Since $\lambda \meg \frac 1 2 m_k= \frac 1 2 m'_{r-k+1}$, we must have $\ell\meg k-1$. Since, in addition, $\lambda \vect 1_r+2\vect s\succ_{\epsb} \frac 1 2 \vect m'^{\epsb}$, we have $s_{r-\ell}=0$, so that  $k\meg \ell $,  which contradicts the preceding condition. 
	
	Therefore, $k=r$, in which case $\frac 1 2 m_{r}=m /r-1$ and the assertion follows by means of Proposition~\ref{cor:1}.
\end{proof}

\subsection{The Circular Bounded Realization of $D$}\label{sec:bounded}

In this subsection, we collect some remarks on the circular bounded realization of $D$ which will be of use when describing the case $n>0$. Cf., e.g.,~\cite{FarautKoranyi2,ArazyFisher2,Arazy} for more information.

Observe that, by~\cite[Chapters 2, 10]{Loos}, there are a circular convex bounded symmetric domain $\Dc$ in $E\times F_\C$ and a birational biholomorphism $\Cc\colon D\to \Dc$ (the (inverse) `Cayley transform') such that the following hold:
\begin{itemize}
	\item there are two rational mappings $\Cc_{F}\colon F_\C\to F_\C$ and  $\Cc_E\colon F_\C\to \Lin(E)$ such that 
	\[
	\Cc(\zeta,z)=(\Cc_E(z)\zeta,\Cc_F(z))
	\]
	for every $(\zeta,z)\in D$;

	\item $\Cc_F(z)=(z+i e_\Omega)^{-1}(z-i e_\Omega)$ for every $z\in T_\Omega$ and $\Cc_F$ induces a birational biholomorphism of $T_\Omega$ onto $\Dc_0\coloneqq\Set{z\in F_\C\colon (0,z)\in \Dc}$.
\end{itemize}

In addition, $\Cc G(D) \Cc^{-1}$ is the group of biholomorphisms $G(\Dc)$ of $\Dc$, so that the isomorphism $G_0(D)\ni \phi\mapsto \Cc \phi\Cc^{-1}\in G_0(\Dc)$ lifts to an isomorphism of $\widetilde G(D)$ onto $\widetilde G(\Dc)$.

For every $\lambda\in \R$, we may then define a representation $\widetilde U_\lambda$ of $\widetilde G(\Dc)$ in $\Hol(\Dc)$ so that
\[
\widetilde U_\lambda(\phi) f=(f\circ \phi^{-1}) (J\phi^{-1})^{\lambda/g}
\]
for every $f\in \Hol(\Dc)$ and for every $\phi\in \widetilde G(\Dc)$, with the same conventions as before. We define a ray representation $U_\lambda$ of $G(\Dc)$ in $\Hol(\Dc)$ analogously. Since $\widetilde U_\lambda$ and $U_\lambda$ are defined in similar ways on $\widetilde G(\Dc)$ and $G(\Dc)$ and on $\widetilde G(D)$ and $G(D)$, we hope that this abuse of notation will not lead to any issues.
If we define an isomorphism $\Cc_\lambda\colon \Hol(D)\to \Hol(\Dc)$ so that
\[
\Cc_\lambda f= (f\circ \Cc^{-1}) (J\Cc^{-1})^{\lambda/g}
\]
for every $f\in \Hol(D)$, then $\Cc_\lambda$ intertwines the two $\widetilde U_\lambda$ (and the two $U_\lambda$), possibly up to a unitary character of $\widetilde G$ (depending on the definition of $(J\Cc^{-1})^{\lambda/g}$).

Now, observe that the stabilizer $\Kc_0$ (resp.\ $\Kc$) of $0$ in $G_0(\Dc)$ (resp.\ $G(\Dc)$) is the group of linear transformations in $G_0(\Dc)$ (resp.\ $G(\Dc)$), cf., e.g.,~\cite[1.5]{Loos}, and is a maximal compact subgroup of $G_0(\Dc)$  (resp.\ $G(\Dc)$). In addition, we have the following result (cf.~\cite[Theorem 2.1]{FarautKoranyi2}).

\begin{prop}\label{prop:30}
	The space of finite $\Kc_0$-vectors\footnote{In other words, the spaces of $f\in \Hol(\Dc)$ whose $\Kc_0$-orbit is finite-dimensional.} (under composition) in $\Hol(\Dc)$ is the space $\Qc$ of holomorphic polynomials on $\Dc$. In addition, for every $\vect s\in \N_\Omega$, the $\Kc_0$-invariant space $\Qc_{\vect s}$ generated by $\Delta^{\vect s}$ is irreducible and $\Kc$-invariant, and  $\Qc=\bigoplus_{\vect s\in \N_\Omega} \Qc_{\vect s}$. 
\end{prop}

\begin{proof}
	All assertions follow from~\cite[Theorem 2.1]{FarautKoranyi2}, except for the $\Kc$-invariance of the $\Qc_{\vect s}$. To see this latter fact, observe first that $G(D)= G_0(D)\Aff(D)$ by~\cite[Remark 1]{Nakajima}, and that $\Aff(D)=K_\Aff G_T$, where $K_\Aff$ is the stabilizer of $(0,i e_\Omega)$ in $GL(D)$. Then, $G(D)=K_\Aff G_0(D)$. It will then suffice to prove that  $\Cc K_\Aff\Cc^{-1}$ preserves the $\Qc_{\vect s}$. Then, take $A\times B_\C\in K_\Aff$, so that $B$ is in the stabilizer of $0$ in $G(\Omega)$, $A\in GL(E)$, and $B_\C\Phi=\Phi(A\times A)$. Then,
	\[
	(\Cc(A\times B_\C )\Cc^{-1})(\zeta,z)=(\Cc_E( B_\C \Cc_F^{-1}(z)) A \Cc_E(\Cc_F^{-1}(z))^{-1}\zeta,   \Cc_F B_\C \Cc_F^{-1}(z))=(A'(\zeta,z),B'(z))
	\]
	for every $(\zeta,z)\in \Dc$, where $A'\in \Lin(E\times F_\C;E)$ and $B'$ is a linear automorphism of $\Dc_0$ (the fact that $A'$ and $B$ must be linear follows from the fact that $\Cc K_\Aff\Cc^{-1}\subseteq \Kc$). Therefore,
	\[
	\Delta^{\vect s}((\Cc(A\times B_\C )\Cc^{-1})(\zeta,z))=\Delta^{\vect s}(B'(z)),
	\]
	for every $(\zeta,z)\in \Dc$.  Now, observe that $\Cc_F z= (z+i e_\Omega)^{-1}(z-i e_\Omega)$ and $\Cc_F^{-1}z=i (z+e_\Omega)(e_\Omega-z)^{-1}$. Since $B$ belongs to the stabilizer $K$ of $e_\Omega$ in $G(\Omega)$, it induces an automorphism of $F$ (as a Jordan algebra, cf.~\cite[p.~56--57]{FarautKoranyi}). Therefore, $B$ commutes with both $\Cc_F$ and $\Cc^{-1}_F$, so that $B'=B_\C$. 
	Thus,  Proposition~\ref{prop:13} shows that $\Delta^{\vect s}\circ B=\sum_j a_j (\Delta^{\vect s}\circ B_j)$, for some  $a_1,\dots, a_N\in\C$ and some $B_1,\dots, B_N\in K_0$.
	Now, Proposition~\ref{prop:15} shows that there are $A_1,\dots, A_N\in GL(E)$ such that $A_1\times (B_1)_\C,\dots, A_N\times (B_N)_\C\in  K_\Aff\cap \Aff_0(D)$. By holomorphy, it then follows that $\Delta^{\vect s}\circ (\Cc(A\times B_\C )\Cc^{-1})=\sum_j a_j  [\Delta^{\vect s}\circ (\Cc(A_j\times (B_j)_\C )\Cc^{-1})]\in \Qc_{\vect s}$, whence the result.
\end{proof}

In particular, if $\chi_{\vect s}$ denotes the character of the irreducible representation of $\Kc_0$ in $\Qc_{\vect s}$ (by composition), then the operators $Q_{\vect s}$ on $\Hol(\Dc)$, defined by
\[
Q_{\vect s}f\coloneqq \int_{\Kc_0} f(k^{-1}\,\cdot\,) \overline{\chi_{\vect s}(k)}\,\dd k,
\]
are  self-adjoint projectors of $\Hol(\Dc)$ onto $\Qc_{\vect s}$ such that $Q_{\vect s} Q_{\vect s'}=0$ if $\vect s\neq \vect s'$ and $I=\sum_{\vect s\in \N_\Omega} Q_{\vect s}$ pointwise on $\Lin(\Hol(\Dc))$.\footnote{To see this latter fact, take $f\in \Hol(\Dc)$, and observe that $Q_{\vect s}[f(R\,\cdot\,)]=(Q_{\vect s} f)(R\,\cdot\,)$ for every $R\in (0,1)$, so that we may reduce to the case in which $f$ is holomorphic on $R D$ for some $R>1$. In this case, $f\in H^2(\Dc)$ and the sum $\sum_{\vect s} Q_{\vect s}f$ converges in $H^2(\Dc)$, hence in $\Hol(\Dc)$, since the $\Qc_{\vect s}$ are pairwise orthogonal in $H^2(\Dc)=\Ac_{(m+n)/r}(\Dc)$ and $Q_{\vect s}$ induces the self-adjoint projector of $H^2(\Dc)$ onto $Q_{\vect s}$, as the discussion below shows.}

In addition, if $\gf$ denotes the Lie algebra of $G(\Dc)$ (identified with the Lie algebra of $\widetilde G(D)$), the derived representation $\dd\widetilde U_\lambda$ of $\widetilde U_\lambda$ preserves $\Qc$ and thus endows $\Qc$ with the structure of a $(\gf,\widetilde \Kc)$-module.\footnote{See, e.g.,~\cite{Adamsetal,Wallach,Wallach2} for more on the theory of $(\gf,\widetilde \Kc)$-modules. Notice, though, that the group $\widetilde G(\Dc)$ is \emph{not} reductive (and that $\widetilde \Kc$ is not compact) in this case, so that the theory developed in the cited references may not be applied directly in this context. The original theory developed by Harish-Chandra does, though.}
In particular, by means of the projectors $Q_{\vect s}$ described above, we see that the mappings $V\mapsto V\cap \Qc$ and $V\mapsto \overline V$ induce two inverse bijections between the set of closed $\widetilde U_\lambda$-invariant subspaces of $\Hol(\Dc)$ and the set of $(\gf,\widetilde \Kc)$-submodules of $\Qc$ (that is, $\gf$-$\dd U_\lambda$-invariant and $\widetilde \Kc$-invariant subspaces of $\Qc$). As a consequence of~\cite[Theorem 5.3]{FarautKoranyi2} and Proposition~\ref{prop:18} below, we then know that the only $(\gf,\widetilde \Kc)$-submodules of $\Qc$ (induced by $\widetilde U_\lambda$) are the $\bigoplus_{q(\vect s,\lambda)\meg j} \Qc_{\vect s}$, where $j=-1,\dots, q(\lambda)\coloneqq\max_{\vect s} q(\vect s,\lambda) $ and $q(\vect s,\lambda)$ is the multiplicity of $\lambda$ as a zero of the function 
\[
\lambda'\mapsto \prod_{k=1}^r  \left(\lambda'-\frac 1 2 m_k\right)\cdots  \left(\lambda'-\frac 1 2 m_k+s_k-1\right)
\]
In particular, with the notation of Proposition~\ref{prop:12}, $q(\lambda)=\card(K_\lambda)$ for every $\lambda\in \R$.

Now, set 
\[
(\vect s)^{\vect s'}\coloneqq \prod_{j=1}^r \prod_{k=0}^{s'_j-1} (s_j+k) \qquad \text{and}\qquad (\vect s)'^{\vect s'}\coloneqq \prod_{j=1}^r \prod_{\substack{k=0,\dots,s'_j-1\\s_j+k\neq 0}} \abs{s_j+k}
\]
for every $\vect s\in \R^r$ and for every $\vect s'\in \N^r$. Then,~\cite[Theorem 3.8]{FarautKoranyi2} shows that, for every $\lambda>m/r-1$,
\[
\Ac_\lambda(\Dc)=\Cc_\lambda(\Ac_\lambda(D))=\Set{f\in \Hol(\Dc)\colon \sum_{\vect s\in \N_\Omega} \frac{1}{\left(\lambda\vect1_r-\frac 1 2 \vect m \right)^{\vect s}} \norm{Q_{\vect s}(f)}_{\Fc}^2<\infty},
\]
with
\[
\norm{f}_{\Ac_\lambda(\Dc)}^2=c_\lambda \sum_{\vect s\in \N_\Omega} \frac{1}{\left(\lambda \vect1_r-\frac 1 2 \vect m \right)^{\vect s}} \norm{Q_{\vect s}(f)}_{\Fc}^2
\]
for every $f\in \Ac_\lambda(\Dc)$, where $\norm{f}_{\Fc}^2=\int_{E\times F_\C} \abs{f(z)}^2 \ee^{-\abs{z}^2}\,\dd z$ for every holomorphic polynomial $f$ on $E\times F_\C$ (cf.~also~\cite[Proposition XI.1.1]{FarautKoranyi}). Then, take $\lambda\in m/r-1-\N$, and define
\[
H_\lambda(\Dc)=\Set{f\in \Hol(\Dc)\colon \sum_{q(\vect s,\lambda)=q(\lambda)} \frac{1}{\left(\lambda \vect1_r-\frac 1 2 \vect m \right)'^{\vect s}} \norm{Q_{\vect s}(f)}_{\Fc}^2},
\]
endowed with the corresponding scalar product. Observe that the closure $V_{\lambda,q(\lambda)}$ of $\bigoplus_{q(\vect s,\lambda)<q(\lambda)} \Qc_{\vect s}$ in $\Hol(\Dc)$ is the closure of $\Set{0}$ in $H_\lambda(\Dc)$, and that $H_\lambda(\Dc)$ embeds continuously into $\Hol(\Dc)/V_{\lambda,q(\lambda)} $. For this latter fact, it will suffice to observe that there is  $C>0$ such that $\left(\lambda\vect 1_r -\frac 1 2 \vect m\right)'^{\vect s} \meg C \left(p\vect 1_r -\frac 1 2 \vect m\right)^{\vect s}$ for every $\vect s\in \N_\Omega$ such that $q(\vect s,\lambda)=q(\lambda)$, so that $H_\lambda(\Dc)$ embeds continuously into $\Ac_p(\Dc)/(\Ac_p(\Dc)\cap V_{\lambda,q(\lambda)})$, which in turn embeds continuously into $\Hol(\Dc)/V_{\lambda,q(\lambda)}$. Thus, $H_\lambda(\Dc)$ is strongly decent and saturated. Since, in addition, the seminorm of $H_\lambda(\Dc)$ is lower semi-continuous for the topology of $\Hol(\Dc)$, we see that $H_\lambda(\Dc)$ is complete, hence a semi-Hilbert space. 

Now,~\cite[Theorem 5.3]{FarautKoranyi2} shows that the scalar product of $H_\lambda(\Dc)$ is $\gf$-$\dd\widetilde U_\lambda$-invariant and $\widetilde \Kc$-$\widetilde U_\lambda$-invariant. Let us now prove that $H_\lambda(\Dc)$ is $\widetilde U_\lambda$-invariant with its seminorm. To this aim, let $\pi\colon \widetilde G(\Dc)\to G_0(\Dc)$ be the canonical projection, so that $\ker \pi$ is a dicrete central subgroup of $\widetilde G(\Dc)$. Observe that there is a unitary character $\chi_\lambda$ of $\ker \pi$ such that $\widetilde U_\lambda(\phi\psi)=\chi_\lambda(\phi)\widetilde U_\lambda(\psi)$ for every $\phi,\psi\in \ker \pi$. More precisely, observe that $\lambda/p$ is a rational number, so that there is $N\in\N^*$ such that $N\lambda/p\in\Z$. Then, $\chi_\lambda^N=1$, so that $ \chi_\lambda^{-1}(1)$ is a subgroup of index at most $N$ of $\ker \pi$. Thus, $\widetilde G(\Dc)/\chi_\lambda^{-1}(1)$ is a finite covering of $G_0(\Dc)$, and $\widetilde U_\lambda$ induces a representation of $\widetilde G(\Dc)/\chi_\lambda^{-1}(1)$ in $\Hol(\Dc)$. In particular, $\widetilde G(\Dc)/\chi_\lambda^{-1}(1)$ is a \emph{real reductive group}, so that~\cite[Corollary 4.24]{Wallach} shows that  $H_\lambda(\Dc)$ is $\widetilde U_\lambda$-invariant with its seminorm.

\subsection{The General Case}

In order to deal with   the case $n>0$, we shall heavily rely on the corresponding results for bounded domains. 

We shall begin with a rather implicit, yet useful, description of the closed $G_0$-$U_\lambda$-invariant subspaces of $\Hol(D)$.

\begin{prop}\label{prop:18}
	Take $\lambda\in \R$ and a closed subspace $V $ of $\Hol(D)$. With the notation of Proposition~\ref{prop:12}, for every $k\in K_\lambda$ define 
	\[
	V_{\lambda,k}\coloneqq \Set{f\in \Hol(D)\colon \forall \phi\in G_0 \:\: [U_\lambda(\phi) f]*I^{-\vect s}=0}
	\]
	where $\vect s$ is any element of $N_{\lambda,k}$.
	Then, $V_{\lambda,k}$ does \emph{not} depend on the choice of $\vect s$, and $V$ is  $G_0$-$U_\lambda$-invariant if and only if it is either $\Set{0}$, $\Hol(D)$, or $V_{\lambda,k}$ for some $k\in K_\lambda$. The space $V$ is then $G$-$U_\lambda$-invariant.
	
	In addition, if $k, k'\in K_\lambda$ and $k\neq k'$, then $V_{\lambda,k}\neq V_{\lambda,k'}$, and $V_{\lambda,k}$ is generated by   $\C\chi_E\otimes   \bigcap_{\vect s\in N_{\lambda,k}}\ker \Dc_{\vect s}$.
\end{prop}

In particular, the invariant spaces considered in the above proposition corresponding to different $k$ are all different, and different from $\Set{0}$ and $\Hol(D)$.

In the bounded realization, the $V_{\lambda,k}$, $k\in K_\lambda$, are simply the closures in $\Hol(\Dc)$ of the $\bigoplus_{q(\vect s,\lambda)\meg j} \Qc_{\vect s}$, $j=0,\dots, q(\lambda)-1$ (cf.~Subsection~\ref{sec:bounded}).

\begin{proof}
	We keep the notation of Subsection~\ref{sec:bounded}. Then, $\Vc\coloneqq \Cc_\lambda(V)$ is a $G_0(\Dc)$-$U_\lambda$-invariant closed subspace of $\Hol(\Dc)$. Let $\Vc_\Kc\coloneqq \Vc\cap \Qc$ be the space of finite $\Kc_0$-vectors in $\Vc$, so that $\Vc=\overline{\Vc_\Kc}$.
	Denote by $\Vc_{\Kc,0}$ the space of restrictions to $\Dc_0$ of the elements of  $\Vc_\Kc$, and by $\Vc_0$ its closure in $\Hol(T_\Omega)$. By~\cite[Theorem 2.1]{FarautKoranyi2}, 	$\Vc_\Kc$ is the  $\Kc_0$-$U_\lambda$-invariant subspace of  $\Hol(\Dc)$ generated by the $\Delta^{\vect s}$, $\vect s\in \N_\Omega $, that it contains. Therefore,	
	$\Vc$ is the closed $\widetilde G(\Dc)$-$\widetilde U_\lambda$-invariant (or simply $\Kc_0$-$U_\lambda$-invariant) subspace of  $\Hol(\Dc)$ generated by $\Set{(\zeta,z)\mapsto f(z)\colon f\in \Vc_{\Kc,0}}$, hence also by $\Set{(\zeta,z)\mapsto f(z)\colon f\in \Vc_{0}}$. 
	Define $V_0\coloneqq \Cc_{F,\lambda}^{-1}\Vc_0$, where $\Cc_{F,\lambda}$ is defined from $\Cc_F$ as $\Cc_\lambda$ is defined from $\Cc$, and set
	\[
	\widetilde U_\lambda^0(\phi)\colon f \mapsto (f\circ \phi^{-1}) (J \phi^{-1})^{\lambda/(2 m/r)}
	\]
	for every $\phi \in \widetilde G(T_\Omega)$. Let us prove that $V_0 $ is $\widetilde G(T_\Omega)$-$\widetilde U_\lambda^0$-invariant. 
	Observe first that, since by Proposition~\ref{prop:15} for every $\phi \in \Aff_0(T_\Omega)$ there is $\psi\in GL(E)$ such that $\psi\times \phi\in \Aff_0(D)$, it is clear that $V_0$ is $\Aff_0(T_\Omega)$-$U_\lambda^0$-invariant. 
	Then, take $\iota$ as in Proposition~\ref{prop:15}, so that $(J\iota)(\zeta,z)=i^{-n}\Delta^{-p\vect 1_r}(z)$ and $(J\iota_0)(z)=\Delta^{-(2m/r)\vect 1_r}(z)$ for every $(\zeta,z)\in D$, where $\iota_0$ is the biholomorphism of $T_\Omega$ induced by $\iota$, thanks to Proposition~\ref{prop:15}. 
	Then, we may identify $\iota$ and $\iota_0$ with suitable elements of $\widetilde G$ and $\widetilde G(T_\Omega)$ in such a way that $\ee^{\lambda n \pi i/(2 p) }(J\iota)^{-\lambda/p}(\zeta,z)=(J\iota_0)^{-\lambda/(2 m/r)}(z)$ for every $(\zeta,z)\in D$, so that $V_0$ is $\widetilde U^0_\lambda(\iota_0)$-invariant. 
	Since $G_0(T_\Omega)$ is generated by $\Aff_0(T_\Omega)$ and $\iota_0$ by Proposition~\ref{prop:15}, this implies that $V_0$ is $\widetilde U^0_\lambda$-invariant. 
	Observe that $V_0\neq \Set{0}, \Hol(T_\Omega)$ since $\Vc_0$ is the closure of $\Vc_{\Kc,0}$ and $\Vc_{\Kc,0}$ is different from $\Set{0}$ and is not dense in the space of holomorphic polynomials on $T_\Omega$ by the preceding analysis. 
	Since $V_0$ is $\widetilde U^0_\lambda$-invariant, and is different from $\Set{0}$ and $\Hol(D)$, Proposition~\ref{prop:12} implies that $V_0=\bigcap_{\vect s\in N_{\lambda,k}}\ker \Dc_{\vect s}$ for some $k\in K_\lambda$. 
	
	It then follows that $V$ is the closed $\widetilde G$-$\widetilde U^0_\lambda$-invariant  subspace of $\Hol(D)$ generated by
	\[
	\C \chi_E\otimes  \bigcap_{\vect s\in N_{\lambda,k}}\ker \Dc_{\vect s}.
	\]
	In addition, for every $f\in V$, the restriction of $f$ to $T_\Omega$ belongs to $V_0$. Applying this fact to the translates of $f$ along $\Nc$, we then see that $f* I^{-\vect s}=0$ for every $\vect s\in N_{\lambda,k}$, so that $V \subseteq V_{\lambda,k}$ by the arbitrariness of $f$ and the $\widetilde U_\lambda$-invariance of $V$, independently of   the $\vect s$ chosen to define $V_{\lambda,k}$. Equality actually holds since both $V$ and $V_{\lambda,k}$ are $\widetilde U_\lambda$-invariant and induce $\bigcap_{\vect s\in N_{k,\lambda}}\ker \Dc_{\vect s}$ by restriction to $T_\Omega$ by Propositions~\ref{prop:15} and~\ref{prop:12}. In particular, by the same argument we see that the definition of $V_{\lambda,k}$ does not depend on the choice of $\vect s\in N_{\lambda,k}$.
	The fact that $V$ is actually $G$-$U_\lambda$-invariant follows 	 from Proposition~\ref{prop:30}.	
	
	In order to complete the proof, it will suffice to prove that there are at least $\card{K_\lambda}$ closed $G_0$-$U_\lambda$-invariant subspaces of $\Hol(D)$ which  are different from $\Set{0}$ and $\Hol(D)$. This follows from~\cite[Theorem 5.3]{FarautKoranyi2}.	
\end{proof}

Recall that $\Ac_{\lambda}$ is $G$-$U_\lambda$-invariant with its norm for every $\lambda\in \Wc(\Omega)$ by Proposition~\ref{prop:19}.

\begin{teo}\label{teo:5}
	Take $\lambda\in \R$.  If $\lambda\in   m /r -1 -\N$, then there is a   strongly decent and saturated semi-Hilbert space $H_\lambda$ of holomorphic functions on $D$ such that the following hold:
	\begin{itemize}
		\item $H_\lambda$ is $G$-$U_\lambda$-invariant with its seminorm;
		
		\item $H_\lambda$ embeds continuously into  $\Ac_{\lambda, m/r-\lambda}$;
		
		\item the canonical mapping $H_\lambda/(H_\lambda\cap \ker \square^{m/r-\lambda})\to \widehat \Ac_{\lambda,m/r-\lambda}$ is  a multiple of an  isometry;
		
		\item $\pr_0 H_\lambda=\C \chi_E\otimes_2 \Ac_{\lambda,m/r-\lambda}(T_\Omega)$ with a  proportional seminorm, where $\pr_0(f)\colon(\zeta,z)\mapsto f(0,z)$.\footnote{Given two Hilbert spaces $X,Y $, we denote by $X\otimes_2 Y$ the tensor product of $X$ and $Y$ endowed with the scalar product defined by $\langle x\otimes y\vert x'\otimes y'\rangle\coloneqq \langle x\vert x'\rangle_X\langle y\vert y'\rangle_Y$ for every $x,x'\in X$ and for every $y,y'\in Y$.}
	\end{itemize}
	
	Conversely, assume that $H$ is a non-trivial strongly decent and  saturated semi-Hilbert space of holomorphic functions on $D$ in which $U_\lambda$ induces a bounded (resp.\ isometric) ray representation of $G_0$.  Then, either one of the following conditions holds:
	\begin{itemize}
		\item[\textnormal{(1)}] $\lambda\in \Wc(\Omega)$ and $H=\Ac_{\lambda}$ with an equivalent (resp.\ proportional) norm;
		
		\item[\textnormal{(2)}]  $\lambda\in m/ r -1 -\N$ and $H=H_\lambda$ with equivalent (resp.\ proportional) seminorms;
	\end{itemize}
\end{teo}

Cf.~\cite{Rango1} for a  description of $H_\lambda$ when $r=1$, and also~\cite{Arcozzietal} for another description of $H_0$ when $r=1$.

Notice that the above result improves~\cite[Theorems 5.2 and 5.3]{Arazy} (for $(r,\lambda)\neq (1,0)$), since it also deals with the case in which the $U_\lambda(\phi)$ are uniformly bounded but not necessarily isometric.

We observe explicitly that proving that $H_\lambda$ has the seminorm induced by $\Ac_{\lambda, m/r-\lambda}$ (up to a constant) is equivalent to proving that it is $\Aff_0$-$\Uc_{\lambda\vect 1_r}$-irreducible (or, equivalently, $\Aff_0$-$U_\lambda$-irreducible). Indeed, one implication follows from Proposition~\ref{prop:4b} and   Lemma~\ref{lem:7}. Conversely, assume that $H_\lambda$ is $\Aff_0$-$\Uc_{\lambda\vect 1_r}$-irreducible. Then, using Schur's lemma (cf., e.g.,~\cite[Corollary 1 to Theorem 1]{Naimark}), the continuity of $\square^{m/r-\lambda}\colon H_\lambda\to \Ac_{2 m/r-\lambda}$,  and Lemma~\ref{lem:7}, we see that $\square^{m/r-\lambda}$ is isometric (up to a constant), so that $H_\lambda$ has the seminorm induced by $\Ac_{\lambda, m/r-\lambda}$ (up to a constant). 

We shall now briefly comment on~\cite[Theorem 5.4]{FarautKoranyi2}. Observe that~\cite[Theorem 5.4]{FarautKoranyi2} and the classical theory of Harish-Chandra modules (cf., e.g.,~\cite[Theorem 2.7]{Adamsetal} and the final discussion of Subsection~\ref{sec:bounded}) imply that $\widetilde U_\lambda$ and $\widetilde U_{2 m/r-\lambda}$ are unitarily equivalent as representations of $\widetilde G$ in $H_\lambda/V_\lambda$ and $\Ac_{2 m/r-\lambda}$, respectively, where $V_\lambda$ denotes the closure of $\Set{0}$ in $H_\lambda$. Notice that this fact follows from Proposition~\ref{prop:4} when $n=0$, that is, $D$ is a tube domain. 
This, in turn, implies that $H_\lambda$ is $G_T$-$\Uc_{\lambda\vect 1_r}$-irreducible, with the aforementioned consequences. 
Unfortunately,~\cite[Theorem 5.4]{FarautKoranyi2} is incorrect for $n>0$. In fact, $\widetilde U_\lambda$ (as a representation of $\widetilde G$ in $H_\lambda$) cannot be equivalent to $\widetilde U_\xi$, as a representation of $\widetilde G$ in $\Ac_\xi$, for any $\xi\in \Wc(\Omega)$. Indeed,~\cite[Theorem 2.1]{FarautKoranyi2} shows that $\Ac_\xi$ contains a $1$-dimensional $\widetilde K $-$\widetilde U_\lambda$-invariant subspace (namely, $\C B^{-\xi\vect 1_r}_{(0, i e_\Omega)}$, which corresponds to the space of constant functions on $\Dc$ with the notation of Subsection~\ref{sec:bounded}), whereas $H_\lambda/V_\lambda$ contains none, unless $n=0$.

\begin{proof}
	We keep the notation of Subsection~\ref{sec:bounded}.
	Take $H$ as in the statement. Observe that, by Proposition~\ref{prop:6}, the closure $V$ of $\Set{0}$ in $H$ is a closed $G_0$-$U_\lambda$-invariant subspace of $\Hol(D)$ and the canonical mapping $H\to \Hol(D)/V$ is continuous. If $V=\Set{0}$, then (1) holds by Proposition~\ref{prop:4b} (or Proposition~\ref{prop:20}). We may then assume that $V\neq \Set{0}$.
	
	Observe that we may assume that $\widetilde U_\lambda$ induces a unitary representation of the stabilizer $\widetilde K$ of $(0, i e_\Omega)$ in $\widetilde G(D)$  in $H$, up to replacing the scalar product of $H$ with the equivalent one
	\[
	(f,g)\mapsto \int_{K_0} \langle U_\lambda(k)f\vert U_\lambda(k) g\rangle_H\,\dd k,
	\]
	where $K_0$ denotes the (compact) stabilizer of $(0,i e_\Omega)$ in $G_0(D)$.\footnote{Notice that this latter scalar product is well defined. First, observe that $\langle U_\lambda(k)f\vert U_\lambda(k) g\rangle_H$ is independent of the chosen representative of $U_\lambda(k)$, provided that the same representative is chosen on both sides of the scalar product. Then, observe that this mapping (of $\phi$) is continuous on $G_0$, since it lifts to a continuous mapping on $\widetilde G(D)$ by~\cite[Proposition 2.14]{Rango1}. } In particular, if we identify  $\T$ with a subgroup of $GL(D)$ acting  on $E$ by multiplication, then $\T\subseteq K_0$ and $H$ and its seminorm are $\T$-$U_\lambda$-invariant (or, equivalently, $\T$-$\Uc_{\lambda \vect 1_r}$-invariant). In particular,
	\[
	\pr_0 f=\int_\T \Uc_{\lambda \vect 1_r}(\alpha)f\,\dd \alpha 
	\]
	for every $f\in \Hol(D)$, so that $\pr_0$ induces a self-adjoint projector of $H$ onto $H\cap (\C \chi_E\otimes \Hol(T_\Omega))$. 
	Now, define $\Hc$ and $\Vc$ as the sets of $f\in \Hol(T_\Omega)$ such that the mapping $(\zeta,z)\mapsto f(z)$ belongs to  $\pr_0(H)$ and  $\pr_0(\Hol(D))$, respectively, so that $\pr_0(H)=\C\chi_E\otimes \Hc$ and $\pr_0(V)=\C \chi_E\otimes \Vc$. If we endow $\Hc$ with the scalar product induced by the bijection $\pr_0(H)\ni f\mapsto f(0,\,\cdot\,)\in \Hc$, then $\Hc$ becomes a semi-Hilbert space such that $H=\C\chi_E\otimes_2 \Hc$, such that $\Vc$ is the closure of $\Set{0}$ in $\Hc$, and such that the mapping $\Hc\to \Hol(T_\Omega)/\Vc$ is continuous (cf.~the proof of~\cite[Proposition 5.1]{Rango1}). In particular, $\Hc$ is strongly decent and saturated.
	Define $U^0_\lambda\colon G(T_\Omega)\to \Lin(\Hol(T_\Omega))/\T$   so that $U^0_\lambda(\phi) f=(f\circ \phi^{-1}) (J\phi^{-1})^{\lambda/(2 m/r)}$ for every $\phi\in G(T_\Omega)$ and for every $f\in \Hol(T_\Omega)$.
	Using Proposition~\ref{prop:15}, one may then show that $U^0_\lambda$ induces a bounded (resp.\ isometric) representation of $G_0$ in $\Hc$.
	
	Observe that Proposition~\ref{prop:18} implies that $V$ is the closed $G_0$-$U_\lambda$-invariant subspace of $\Hol(D)$ generated by $\pr_0(V)$, so that $\Vc\neq \Set{0}$. 
	Then, Theorem~\ref{teo:4} implies that $\lambda\in m/r-1-\N$, that  $\Vc=\ker \square^{m/r-\lambda}$, and that $\Hc=\Ac_{\lambda,m/r-\lambda}(T_\Omega) $ with an equivalent (resp.\ proportional) seminorm. 
	In addition, Proposition~\ref{prop:18} implies that $V\subseteq \ker \square^{m/r-\lambda}$, so that Proposition~\ref{cor:1} implies that $H\subseteq \Ac_{\lambda,m/r-\lambda}$ continuously, and that the canonical mapping $H/(H\cap \ker \square^{m/r-\lambda})\to \widehat \Ac_{\lambda, m/r-\lambda}$ is an isomorphism (resp.\ a multiple of an isometry).

	Since $\widetilde U_\lambda$ induces a unitary representation of $\widetilde K$ in $H$, by the arguments of Subsection~\ref{sec:bounded}  we know that the projectors $Q_{\vect s}$ on $\Hol(\Dc)$, transferred to projectors $Q'_{\vect s}=\Cc_\lambda^{-1}Q_{\vect s}\Cc_\lambda$ on $\Hol(D)$, are self-adjoint on $H$, so that  the orthogonal direct sum of the $Q'_{\vect s}(H)$ is dense in $H$.\footnote{When $q(\vect s,\lambda)<q(\lambda)$, this follows from the fact that $Q'_{\vect s}(H)\subseteq V$ by the analysis of Subsection~\ref{sec:bounded}.} Since, in addition, $V$ is the largest proper $\widetilde U_\lambda$-invariant closed subspace of $\Hol(D)$ by Proposition~\ref{prop:18}, we see that $H$ is dense in $\Hol(D)$, so that $Q'_{\vect s}(H)=\Qc'_{\vect s}\coloneqq \Cc_\lambda^{-1}(\Qc_{\vect s})$ for every $\vect s\in \N_\Omega$.
	
	Now, set (cf.~Subsection~\ref{sec:bounded})
	\[
	H_\lambda(D)\coloneqq \Cc_\lambda^{-1}H_\lambda(\Dc)=\Set{f\in \Hol(D)\colon \sum_{q(\vect s,\lambda)=q(\lambda)} \frac{1}{\left(\lambda \vect 1_r-\frac 1 2 \vect m\right)'^{\vect s}} \norm{Q'_{\vect s}f}_{\Cc_\lambda^{-1}\Fc}^2<\infty},
	\]
	so that $H_\lambda(D)$ is   a non-trivial strongly decent and saturated semi-Hilbert space  of holomorphic functions on $D$ which is   $\widetilde U_\lambda$-invariant with its seminorm. Then, the preceding analysis shows $\pr_0 H_\lambda=\pr_0 H=\C\chi_E\otimes_2\Ac_{\lambda,m/r-\lambda}$ with equivalent (resp.\ proportional) seminorms, so that there are constants $C\Meg 1$ (resp.\ $C=1$) and $C'>0$ such that
	\begin{equation}\label{eq:1}
	\frac 1 C\norm{f}_{H}\meg C'\norm{f}_{H_\lambda(D)}\meg C \norm{f}_H
	\end{equation}
	for every $f\in \C\chi_E\otimes_2\Ac_{\lambda,m/r-\lambda}$. In particular, this shows that~\eqref{eq:1} holds	for every $f\in \pr_0(\Qc'_{\vect s})$ and for every $\vect s\in\N_\Omega$. Now, observe that each $\Qc'_{\vect s}$ is $K_0$-$U_\lambda$-irreducible, so that it admits only one $K_0$-$U_\lambda$-invariant norm, up to a multiplicative constant. Since $ \pr_0(\Qc'_{\vect s})\neq \Set{0}$ (for example, $\Cc_\lambda^{-1}(\Delta^{\vect s})\in \Qc'_{\vect s}$), and since both $H$ and $H_\lambda$ induce $K_0$-$U_\lambda$-invariant seminorms on $\Qc_{\vect s}'$, the above analysis shows that~\eqref{eq:1} holds	for every $f\in  \Qc'_{\vect s}$ and for every $\vect s\in\N_\Omega$. Since the $\Qc'_{\vect s}$ are pairwise orthogonal in both $H$ and $H_\lambda(D)$, and their sum is dense in both $H$ and $H_\lambda(D)$ by the preceding analysis, this proves that $H=H_\lambda(D)$ with equivalent (resp.\ proportional) seminorms.  
	
	It only remains to prove that $H_\lambda(D)$ is $G$-$U_\lambda$-invariant with its seminorm. Since, however, each $\Qc'_{\vect s}$ is $K$-$U_\lambda$-invariant with its norm by Proposition~\ref{prop:30}, and since $G(D)=G_0(D) K$, the assertion follows.
\end{proof}

\end{document}